\documentclass[a4paper]{article}

\usepackage{amsmath,amssymb,mathtools,amsthm}
\usepackage{xcolor,graphicx,float,enumitem}
\usepackage{bm}
\graphicspath{{figures/}}

\usepackage[default,scale=1]{opensans}
\usepackage[T1]{fontenc}

\usepackage[utf8]{inputenc}
\DeclareUnicodeCharacter{00A0}{ }
\usepackage{geometry}
\definecolor{purple}{rgb}{0.65, 0, 1}
\definecolor{orange}{rgb}{0.95,.35,0}

\newcommand{\ee}{\varepsilon}

\newcommand{\defeq}{{\coloneqq}}

\newcommand{\vertiii}[1]{{\left\vert\kern-0.25ex\left\vert\kern-0.25ex\left\vert #1
		\right\vert\kern-0.25ex\right\vert\kern-0.25ex\right\vert}}

\newcommand{\vertll}[1]{{\left[\kern-0.1ex\left[ #1
		\right]\kern-0.1ex\right]}}

\DeclareMathOperator{\N}{N}

\newcommand{\Rd}{{\mathbb R^n}}

\newcommand{\diver}{\nabla \cdot}

\newcommand{\vv}[1]{\mathbf{#1}}

\newcommand*\diff{\mathop{}\!\mathrm{d}}

\newcommand{\rhos}{{\widetilde \rho}}
\newcommand{\Ws}{{\widetilde W}}
\newcommand{\Es}{{\widetilde E}}

\usepackage[hidelinks]{hyperref}
\usepackage[nameinlink]{cleveref}

\newtheorem{theorem}{Theorem}[section]
\newtheorem{corollary}[theorem]{Corollary}%
\newtheorem{lemma}[theorem]{Lemma}%

\theoremstyle{definition}
\newtheorem{remark}[theorem]{Remark}%

\numberwithin{equation}{section}

\usepackage[
url=false,
isbn=false,
maxcitenames=4,
giveninits,
maxbibnames=100,
doi=true,
uniquename=false,
block=none]{biblatex}
\renewbibmacro{in:}{}
\DeclareFieldFormat[article]{citetitle}{#1}
\DeclareFieldFormat[article]{title}{#1}

\makeatletter
\renewcommand*{\@fnsymbol}[1]{\ensuremath{\ifcase#1\or \star \or \dagger\or \ddagger\or
		\mathsection\or \mathparagraph\or \|\or **\or \dagger\dagger
		\or \ddagger\ddagger \else\@ctrerr\fi}}
\makeatother

\usepackage{todonotes}

\title{Asymptotic simplification of Aggregation-Diffusion equations towards the heat kernel}
\author{Jos\'e A. Carrillo%
\thanks{Mathematical Institute, University of Oxford, Oxford OX2 6GG, UK.  \href{mailto:carrillo@maths.ox.ac.uk}{carrillo@maths.ox.ac.uk}} %
\and David Gómez-Castro%
\thanks{Mathematical Institute, University of Oxford, Oxford OX2 6GG, UK.  \href{mailto:gomezcastro@maths.ox.ac.uk}{gomezcastro@maths.ox.ac.uk}}
\and Yao Yao%
\thanks{
School of Mathematics, Georgia Institute of Technology. Atlanta, GA 30332, USA.  yaoyao@math.gatech.edu
}
 \and Chongchun Zeng%
\thanks{
School of Mathematics, Georgia Institute of Technology. Atlanta, GA 30332, USA.  chongchun.zeng@math.gatech.edu
}
} 
\begin{document}
	
	\maketitle

\begin{abstract}
We give sharp conditions for the large time asymptotic simplification of aggregation-diffusion equations with linear diffusion. As soon as the interaction potential is 
bounded and its first and second derivatives decay fast enough at infinity,
then the linear diffusion overcomes its effect, either attractive or repulsive, for large times independently of the initial data, and solutions behave like the fundamental solution of the heat equation with some rate.
The potential $W(x) \sim  \log |x|$ for $|x| \gg 1$ appears as the natural limiting case when the intermediate asymptotics change.
In order to obtain such a result, we produce uniform-in-time estimates in a suitable rescaled change of variables for the entropy, 
the second moment, 
Sobolev norms and the $C^\alpha$ regularity with a novel approach for this family of equations using modulus of continuity techniques.
\end{abstract}

\textbf{Keywords:} aggregation-diffusion, intermediate asymptotics, decay rates

\textbf{MSC (2020):} 
35B40,  %
35B45,  %
35C06,  %
35Q92  %

\section{Introduction}
In this work, we analyse the long time asymptotics for probability density solutions to the general aggregation-diffusion equation of the form
\begin{equation}
	\label{eq:ADE rho}
	\tag{P}
	\frac{\partial\rho}{\partial t} = \Delta \rho + \diver (\rho \nabla W * \rho ),
\end{equation}
with $W:\Rd\to \mathbb{R}$ being the interaction potential which is assumed to be symmetric $W(x) = W(-x)$. The assumption of unit mass is not restrictive up to a change of variables due to the (formal) conservation of mass. This work is devoted to identify sharp conditions on the interaction potential $W$ such that the intermediate asymptotic behaviour of the solutions to \eqref{eq:ADE rho} is given by the heat kernel, 
\[
K(t,x) = (2t)^{-\frac n 2} G\left(\frac x{\sqrt{2t}}\right), \qquad \text{where } G(y) = (2\pi)^{-\frac n 2} e^{-\frac {|y|^2} 2 }.
\]
More precisely, our goal is to find the best possible conditions on $W$ such that
\begin{equation}\label{eq:asimp}
    \| \rho(t,\cdot) - K(t,\cdot) \|_{L^1} \to 0, \qquad \text{ as } t \to \infty,
\end{equation}
and if possible, recover the optimal decay rates of the heat equation. Aggregation-diffusion equations of the form \eqref{eq:ADE rho} with linear or nonlinear diffusion are ubiquitous in the literature due to the large number of applications in mathematical biology and mathematical physics, we refer to \cite{CCY19} and the references therein for a recent survey of related results. In the case of linear diffusion, \eqref{eq:ADE rho} is usually referred as the McKean-Vlasov equation associated to a nonlinear SDE process via the mean field limit \cite{JW17,BJW19}. 

The asymptotic simplification of \eqref{eq:ADE rho} can be understood as the case in which the long-time asymptotics of McKean-Vlasov equations is dominated by the linear diffusion term leading to self-similar diffusive behavior for large times. Notice that this result is not true for instance for singular attractive potentials as the Keller-Segel model for chemotaxis or its variants in the diffusion dominated regime \cite{BDP06,BCC12,CCV15,CHVY19,CHMV18} or for McKean-Vlasov equations where the potentials may lead to phase transitions as in \cite{Tu13,BCCD16}. In these cases, there are non-trivial stationary states of the equation \eqref{eq:ADE rho} that attract the long-time dynamics for certain initial data.

Therefore, finding the sharpest conditions on $W$ such that the asymptotic simplification occurs is a challenging question. Notice that even for bounded interaction potentials $W$, the mere time decay of $L^p$ norms, $1<p\leq \infty$, was not known for general initial data. We also extend previous results of \cite{Canizo+Carrillo+Schonbek2012} in which strong integrability assumptions on $W$ and $\nabla W$ were imposed, as well as smallness conditions on $\rho_0$. 
In \cite{GK10} the authors study the case $n = 1$ with $\nabla W \in L^1$ showing \eqref{eq:asimp} without rate for general initial datum.

There is a long literature devoted to the intermediate asymptotics of convection-diffusion equations. 
Results for the heat equation ($W=0$) can be recovered directly from the heat kernel representation (see, e.g.  \cite{Vazquez2017}). 
Better decay rates can be deduced by cancellation of higher order moments, as presented in
\cite{duoandikoetxea1992moments}. 
In \cite{AMTU2001} the authors introduce entropy dissipation arguments through the logarithmic Sobolev inequality that work for a large array of diffusion problems (see also \cite{To99}).
In \cite{ACK08} this method is applied to the heat equation, to recover improved decay rates.
This technique was also used in \cite{Bedrossian11} to recover similar results for \eqref{eq:ADE rho} to recover decay rates, even when the linear diffusion $\Delta \rho$ is replaced by the nonlinear diffusion $\Delta A(\rho)$.
In \cite{EZ91} the authors study the case where the convection is of the type $a \cdot \nabla (|u|^{q-1} u)$. 

As mentioned above, when $W$ has certain growth at infinity there is no decay, and, in fact, $\rho(t)$ converges to an stationary solution which can be recovered by minimisation of free-energy functional (see \cite{CDF19}).
The key example, as we will discuss below, is $W(x) = \chi \ln|x|$ ($\chi >0$), known as the Keller-Segel problem in $n=2$. 
In \cite{BDEF10} the authors discuss the case where $\chi$ is smaller than a critical value, and prove there exists an asymptotic profile different from the Gaussian. This result also holds for any other $n$ (see \cite{BCC08}).
A variation of this problem is studied \cite{LS07} also for small $\chi$. 
When $\chi$ is larger than the critical value, solutions may produce a Dirac delta in finite time.

To analyse the intermediate asymptotics of \eqref{eq:ADE rho}, one classically works in rescaled variables \cite{To99,Carrillo+Toscani2000,AMTU2001}. Following the parabolic scaling of the heat kernel, we 
introduce the new variables $\tau :=\log \sqrt{2t+1}$ and $y:=\frac{x}{\sqrt{2t+1}}$, and
consider a rescaled density $\rhos$ given by 
\begin{equation}
\label{eq:rhos}
    \rhos (\tau,y ) = e^{n \tau} \rho \Big (  \tfrac 1 2 (e^{2\tau} - 1), e^\tau y \Big ).
\end{equation}
The first key element is the existence of a PDE for $\rhos$. It was shown in \cite{Canizo+Carrillo+Schonbek2012} that
we can write
\begin{equation}
\label{eq:Fokker-Plack rho scaled}
	\frac{\partial \rhos}{\partial \tau}  =  \Delta_y \rhos + \nabla_y\cdot (y \rhos ) +  \nabla_y \cdot ( \rhos \nabla_y  (\Ws * \rhos)  ), \qquad \text{where } 	\Ws (\tau, y) := W (e^\tau y).
\end{equation}
Notice that, even though
$\| \widetilde W(\tau,\cdot) \|_{L^1} = e^{-n\tau} \| W \|_{L^1}$ is decaying in $\tau$ if $W\in L^1(\Rd)$,
the norms of $\nabla \Ws,\Delta \Ws$ can exponentially grow in $L^p (\Rd)$ depending on $n$ and $p$. We still expect, under some assumptions, this last term to vanish asymptotically to recover the steady-state of the usual Fokker-Planck equation. However, one cannot directly use classical energy estimates to prove uniform-in-time $L^p$ bounds of $\rhos$ without any additional smallness assumptions on $W$ or $\rho_0$.

First, we obtain a result of global existence and instant regularisation by standard techniques. We then introduce a new estimate on the variational structure of the equation to prove uniform-in-time bounds of natural quantities for the problem such as the second moment, the energy and the entropy. As a consequence, we also show uniform-in-time propagation of the $L^2$, $H^1$ and $C^\alpha$ norms.

\begin{theorem}
    \label{thm:1}
    Let $W \in \mathcal W^{1,\infty} (\Rd)$ 
    and $\rho_0 \in L^1_+ (\mathbb R^n)$ with unit mass.
    Then, there exists a unique mild solution of \eqref{eq:ADE rho} such that $\rho \in C([0,\infty) ; L^1_+ (\Rd)) \cap C((0,+\infty); W^{k,p} (\Rd))$ for all $p \in [1,\infty], k \in \Rd$ with $\rho(t)$ also of unit mass for $t\ge 0$.
    Assume, furthermore, 
    $W(x) = W(-x)$ and
    the initial datum has bounded second order moment and entropy
    \begin{equation}
        \label{eq:rho0 first moment bound}
        \int_{\mathbb R^n} |x|^2 \rho_0 (x) \diff x < \infty, \qquad \int_{\mathbb R^n} \rho_0 \log \rho_0 < \infty.
    \end{equation}
    Then the rescaled density $\rhos$ satisfies
    \begin{equation*}
        \sup_{\tau \ge 0} \int_{\mathbb R^n} |y|^2 \rhos (\tau,y) \diff y < \infty, \qquad \sup_{\tau \ge 0} \int_{\mathbb R^n} \rhos(\tau,y) |\log \rhos (\tau,y)| \diff y < \infty.
    \end{equation*}
    Moreover,  
    \begin{enumerate}
    \item If $\nabla W \in L^n (\Rd)$ then
    \begin{equation*}
        \sup_{\tau \ge 1} \| \rhos(\tau, \cdot) \|_{H^1 } < \infty.
    \end{equation*}
    \item If $n \ge 2$, $\nabla W \in L^n (\Rd)$ and $\Delta W \in L^{\frac n2} (\Rd)$ then
    \begin{equation*}
        \sup_{\tau \ge 1} \| \rhos(\tau, \cdot) \|_{C^\alpha } < \infty, \qquad \forall \alpha \in (0,1).
    \end{equation*}
    \end{enumerate}
    \end{theorem}

The well-posedness and instant regularisation parts of the proof of \Cref{thm:1} is based on the study of the Duhamel formula for
\begin{equation}
    \label{eq:heat divergence data}
    \tag{P$_F$}
    \frac{\partial u}{\partial t} - \Delta u = \diver F.
\end{equation}
Existence and uniqueness are proven by a fixed-point argument. Regularity is achieved by a bootstrap argument in fractional Sobolev spaces. 
The details are presented in \Cref{sec:well-posedness}.
To this end, we develop a new Young inequality for fractional spaces that we present in \Cref{sec:fractional Young-Sobolev inequality}.

In order to recover the propagation of regularity, we take advantage of a second key fact: a sharp decay of the free energy, that leads to a uniform-in-time entropy bound in rescaled variables \eqref{eq:rhos}. This is the objective of \Cref{sec:decay free energy}.
More precisely, since $W(x) = W(-x)$, problem \eqref{eq:ADE rho} is the 2-Wassertein flow associated to the free energy
\begin{equation}\label{def_E}
	E(t)=E[\rho(t)] = \int_{\mathbb R^n} \left (  \rho \log \rho + \frac 1 2 \rho (W * \rho)  \right) \diff x  =   \int_{\mathbb R^n}  \rho \log ( \rho e ^{ \frac 1 2 W * \rho } )  \diff x.
\end{equation}
When $W \in L^\infty(\Rd)$, we prove the sharp decay of the energy $E[\rho]$ in \Cref{thm:free energy decay}. It implies that 
\[E(t) \le -\frac{n}{2}\ln t + C(n, \|W\|_{L^\infty}),\]
which tends to $-\infty$ with the same rate $-\frac{n}{2}\ln t$ as for the heat equation.
We next show that there is a suitable free-energy-like quantity that is bounded below in rescaled variables \eqref{eq:rhos}, and hence we will be able to estimate the second moment. Through the second moment and the free energy, we are able to show uniform-in-time equi-integrability in the form
\begin{equation}\label{eq_integrability}
    \sup_{\tau \ge 0} \int_{\Rd} \rhos | \log \rhos | \diff y < \infty .
\end{equation}
The uniform-in-time propagation of regularity is analysed in \Cref{sec:prop regularity}, 
where \eqref{eq_integrability} is used in a crucial way.
For the uniform-in-time bounds of the $H^1$ norm we apply a standard energy estimate. For the propagation of the $C^\alpha$ norm we apply a modulus of continuity argument, which to our knowledge is new for \eqref{eq:ADE rho} but has been applied successfully for other equations (see, e.g., \cite{Kiselev2007,Kiselev2008}).

Equipped with all these uniform-in-time estimates, we can finally characterise, in \Cref{sec:intermediate asymptotics}, the intermediate asymptotics of the solutions of \eqref{eq:ADE rho}.
\begin{theorem}
\label{thm:decay}
    Let $W \in \mathcal W^{1,\infty} (\Rd) \cap L^1 (\Rd)$ such that $W(x) = W(-x)$,  $\nabla W \in L^n (\Rd)$ and, if $n \ge 2$, also that $\Delta W \in L^{\frac n 2} (\Rd)$. Assume that $\rho_0 \in L^1_+ (\Rd)$ is such that it satisfies \eqref{eq:rho0 first moment bound}. Then we have
     \begin{equation*}
        \| \rho(t, \cdot) - K(t, \cdot) \|_{L^1} \le \begin{cases}
        C t^{-\frac 1 4} & \text{if } n=1,\\
        C t^{-\frac 1 2} (1 + \log(1+2t))^{\frac 1 2}  & \text{if } n=2,\\
        C  t^{-\frac 1 2} & \text{if } n\geq3.
        \end{cases}
    \end{equation*}
\end{theorem}

This result is obtained by a classical entropy-entropy dissipation argument. Following the idea in \cite{To99,Carrillo+Toscani2000,AMTU2001,Canizo+Carrillo+Schonbek2012}, 
we measure the distance between the rescaled version of the solution $\rhos$ and the Gaussian in 
the $L^1$ relative entropy defined by
\begin{equation}
\label{eq:L1 entropy}
	E_1 (\rhos \| G) = \int_{\Rd} \rhos \log \frac{\rhos}{G} \diff y = \int_\Rd \rhos \log \rhos + \frac 1 2 \int_\Rd |y|^2 \rhos \diff y + \frac n 2 \log ( 2 \pi).
\end{equation}
As in the case of the heat equation, this functional can be differentiated in $\tau$. 
We apply the logarithmic Sobolev and Csiszar-Kullback inequalities to reduce ourselves to estimate the remainder terms with respect the classical heat equation due to $\Ws$. We also discuss the $L^2$ relative entropy and the related $L^2$ intermediate asymptotics (see \Cref{sec:L2 rel entropy}).

For $n \ge 3$ the decay rate coincides with that of the heat equation under our assumptions, and hence it seems sharp as a generic rate.
Notice that it is a simple computation that
\begin{equation*}
    \| K(t,\cdot+a) - K(t,\cdot)  \|_{L^1} \sim t^{-\frac 1 2}.
\end{equation*}
Better decay rates for the heat equation can be obtained by correctly matching higher moments, as shown in \cite{duoandikoetxea1992moments} (see also the survey \cite{Vazquez2017} for a clear explanation). For $n \le 2$, we do not expect better rates with our technique, as explained in \Cref{sec:L1 relative entropy}. It is an open problem to improve these rates in $n=1,2$ possibly under stronger assumptions on $W$.

We also answer in \Cref{sec:intermediate asymptotics} the question on minimal assumptions on $W$ such that the asymptotic simplification of the system happens with arbitrarily slow rate.

\begin{theorem}
Let $n\ge 2$, $ W \in \mathcal W^{1,\infty} (\Rd)$ such that $W(x) = W(-x)$, $\nabla W \in  L^{n-\ee} (\Rd)$, $\Delta W \in  L^{\frac n 2} (\Rd)$ (and also $\Delta W \in  L^{\frac n 2 - \ee} (\Rd)$ if $n\geq 3$) for some $\ee > 0$, and that $\rho_0 \in L^1_+ (\Rd)$ is such that it satisfies \eqref{eq:rho0 first moment bound}. Then
$ \| \rho(t, \cdot) - K(t, \cdot) \|_{L^1} \to 0$ as $t\to\infty$.
\end{theorem}
This theorem also works for $n = 1$ under suitable assumptions on $W$ (see \Cref{thm:convergence to the Gaussian L1}). 
Lastly, let us discuss the assumptions $\nabla W \in L^n$ (and $\Delta W \in L^{\frac n 2}$ if $n\ge 2$). A borderline case outside these assumptions is the key example alluded above, $W(x) = \chi  \ln|x|$. The rescaling leads to $\Ws = \chi \ln |y| + \chi \tau $, so $\nabla \Ws$ does not evolve in time. It is easy to see that any solution of
\begin{equation*}
    \ln \rhos + \frac{|y|^2}{2} + \chi \ln |\cdot|  * \rhos = C,
\end{equation*}
for some constant $C$, is a stationary solution for the Fokker-Planck equation \eqref{eq:Fokker-Plack rho scaled} with $\nabla \Ws = \nabla W$, and so the corresponding $\rho$ in original variables is of self-similar form with profile $\rhos$. The existence and uniqueness of these self-similar solutions for subcritical values of $\chi$ was proven in \cite{BDP06,BCC08} by variational methods, moreover they are the intermediate asymptotics for subcritical $\chi$. This explains to some extent how the hypotheses on $W$ are almost sharp for the asymptotic simplification towards the heat kernel profile.

\section{Well-posedness and regularity}
\label{sec:well-posedness}
We make use of the classical approach using Duhamel's formula to obtain sharp well-posedness global in time results, under the assumptions specified below on the potential $W$ 
(see similar results in \cite{Canizo+Carrillo+Schonbek2012}). 
In this section, we use a sub-index $t$ to denote the time variable. 
Using the variation of constants formula we can re-write the problem \eqref{eq:ADE rho} as a fixed-point problem of the form
\begin{align*}
    \rho_t = G_t * \rho_0 + \int_0^t G_{t-s} * \diver \left(  \rho_s  \nabla ( W * \rho_s ) \right) \diff s, 
\end{align*}
or, equivalently, moving derivatives in the convolution 
\begin{equation}
    \label{eq:ADE mild solution}
    \rho_t = G_t * \rho_0 + \int_0^t ( \nabla G_{t-s} ) *   \left ( \rho_s   \nabla ( W * \rho_s ) \right) \diff s. 
\end{equation}
For  two vector fields $F$ and $\overline F$, we denote the component-wise convolution
$
    F*\overline F = \sum_{i=1}^n F_i * \overline F_i .
$
The corresponding formula for \eqref{eq:heat divergence data} is
\begin{equation}
\label{eq:Duhamel u}
    u (t; F) = G_t * \rho_0 + \int_0^t ( \nabla G_{t-s} ) *   F(s) \diff s .
\end{equation}

Below we collect several estimates for the solution $u$ in \eqref{eq:Duhamel u}.

\paragraph{\texorpdfstring{$L^1$}{L1} estimates for \texorpdfstring{$u(t;F)$}{u(t;F)}.} Let us start by obtaining direct basic $L^1$ estimates for $u(t; F)$. We begin by recalling some properties of the heat kernel $G_t$.
Clearly $\|G_t \|_{L^1} = 1$. For integer derivatives
\begin{align*}
    \int_\Rd |D^k G_t|^p \diff x = (2t)^{-\frac{np + kp}{2}}\int_\Rd  |D^k G(\tfrac{x}{\sqrt{2t}})|^p \diff x =  (2t)^{\frac{n - (n+k)p}{2}}\int_\Rd  |D^k G(y)|^p \diff y.
\end{align*}
Then, $\|D^k G_t\|_{L^p}$ is integrable in time for $t$ near zero as long as $p < \frac{n}{n+k-2}$. 
A similar scaling holds in the range of fractional Sobolev spaces $\mathcal W^{s,p}$ (defined in \Cref{sec:fractional Young-Sobolev inequality}) by applying the classical computations presented in \Cref{sec:scaling of fractional Sobolev norm}. In particular,
\begin{equation*}
    \|\nabla G_t \|_{\mathcal W^{s,p}} \le C t^{\frac n {2p} - \frac{n+s}{2}}.
\end{equation*}
With these estimates, we can directly recover $L^1$ estimates for $u(t;F)$ by using Young's inequality
\begin{equation}
\label{eq:heat L1 bound}
\begin{aligned} 
    \|u (t; F) \|_{L^1} &\le \|\rho_0 \|_{L^1} + \int_0^t  \sum_{i=1}^n \left \| \frac{ \partial G_{t-s} }{\partial x_i} \right\|_{L^1} \|F_i (s)\|_{L^1} \diff s \\ 
    &\le \|\rho_0 \|_{L^1} + C \sum_{i=1}^n \sup_{s \in [0,t]} \|F_i (s)\|_{L^1} \int_0^t   (t-s)^{-\frac 1 2}  \diff s 
    \le \|\rho_0 \|_{L^1} + C t^{\frac 1 2} \sum_{i=1}^n \sup_{s \in [0,t]} \|F_i (s)\|_{L^1}.
\end{aligned}
\end{equation}

\paragraph{Continuous dependence with respect to $F$.} Similarly, we can also state a result of {continuous dependence with respect to $F$}
\begin{align*}
    \| u(t; F) - u(t;\overline F) \|_{L^1} &= \left\| \int_0^t ( \nabla G_{t-s} ) *   (F(s) - \overline F(s)) \diff s  \right\|_{L^1} \\
    &\le C \int_0^t (t-s)^{-\frac 1 2} \diff s \sum_{i=1}^n\sup_{t\in[0,t]} \| F_i (s) - \overline F_i (s) \|_{L^1} 
    \le C  t^{\frac 1 2}\sup_{t\in[0,t]} \| F_i (s) - \overline F_i (s) \|_{L^1} .
\end{align*}
Computing the supremum, we recover
\begin{equation}
    \label{eq:heat continuous dependence F}
    \sup_{t \in [0,T]} \| u(t; F) - u(t;\overline F) \|_{L^1} \le C T^{\frac 1 2} \sup_{t \in [0,T]}\| F (s) - \overline F (s) \|_{L^1}.
\end{equation}

\paragraph{Modulus of continuity in time.}  We claim that, if $F$ has a modulus of continuity in $C([0,T];L^1(\Rd))$, it is preserved for $u(t;F)$.
We already know that, for $t > s$
\begin{equation*}
    \|G_t * \rho_0 - G_s *\rho_0\|_{L^1} = \| G_s * (G_{t-s} * \rho_0 - \rho_0) \|_{L^1} \le \|G_{t-s} *  \rho_0 - \rho_0 \|_{L^1} =: \omega_G(t - s; \rho_0).
\end{equation*}
This last element is a modulus of continuity, by the classical result of strong convergence of convolutions.
For the continuity of the second term in \eqref{eq:Duhamel u}, we can write
\begin{align*}
    &\int_0^t ( \nabla G_{t-\tau} ) *   F(\tau) \diff \tau - \int_0^s ( \nabla G_{s-\tau} ) *   F(\tau) \diff \tau  \\
    & \qquad =  \int_0^t ( \nabla G_{t-\tau} ) *   F(\tau) \diff \tau - \int_{t-s}^t ( \nabla G_{t-\tau} ) *   F(\tau - (t-s)) \diff \tau  \\
    &\qquad = \int_0^{t-s} ( \nabla G_{t-\tau} ) *   F(\tau) \diff \tau + \int_{t-s}^t ( \nabla G_{t-\tau} ) * \Big ( F(\tau) -  F(\tau - (t-s)) \Big) \diff \tau.
\end{align*}
On the one hand, we can compute that
\begin{align*}
    \left\| \int_0^{t-s} ( \nabla G_{t-\tau} ) *   F(\tau) \diff \tau \right\|_{L^1} &\le C \sum_{i=1}^n \sup_{\tau \in [0,T]} \|F_i(\tau)\|_{L^1} \int_0^{t-s} (t-\tau)^{-\frac 1 2} \diff \tau \\
    &=  C \left( \sqrt t - \sqrt s \right) \sum_{i=1}^n\sup_{\tau \in [0,T]} \|F_i(\tau)\|_{L^1} 
    \le C \sqrt{t-s}.
\end{align*}
On the other hand, letting $\omega_{F,T}$ be the modulus of continuity of $F$ on $[0,T]$ to $L^1$ we have that
\begin{align*}
    \left\|\int_{t-s}^t ( \nabla G_{t-\tau} ) * \Big ( F(\tau) -  F(\tau - (t-s)) \Big) \diff \tau \right\|_{L^1} &\le C \int_{t-s}^t (t-\tau)^{-\frac 1 2} \sum_{i=1}^n\| F_i(\tau) - F_i(\tau - (t-s))\|_{L^1} \\
    &\le C \omega_{F, T}(t-s) \int_{t-s}^t (t-\tau)^{-\frac 1 2} \diff \tau 
    = C\omega_{F,T}(t-s) \sqrt{s}.
\end{align*}
Hence, Duhamel's formula preserves the continuity, in the sense that
\begin{equation}
\label{eq:modulus continuity heat}
    \| u(t) - u(s) \|_{L^1} \le C\left(  \omega_G(t-s; \rho_0) + \sqrt{t-s} + \sqrt T \omega_{F,T} (t-s) \right) .
\end{equation}

\paragraph{\texorpdfstring{$L^p$}{Lp} estimates for \texorpdfstring{$u(t;F)$}{u(t;F)}.} 
The final result that we need is about the regularisation between {$L^p$ spaces}.
Following a similar procedure as above, we can write
\begin{equation*}
    \| \nabla  G_{t-s} * F (s) \|_{L^p} \le  C \| \nabla  G_{t-s} \|_{L^p} \| F (s) \|_{L^1} \le C  \| F (s) \|_{L^1}(t - s)^{\frac n {2p} - \frac{n+1}{2}},
\end{equation*}
where $(t - s)^{\frac n {2p} - \frac{n+1}{2}}$ is locally integrable in $t$ if $p < \frac{n}{n-1}$. Thus
\begin{equation*}
    F \in C([0,T]; L^1) \implies u (\cdot; F) \in C([\delta, T]; L^p ), \qquad p < \tfrac{n}{n-1}.
\end{equation*}
Analogously, we have
\begin{equation*}
    \| \nabla  G_{t-s} * F (s) \|_{L^r} \le  C \| \nabla  G_{t-s} \|_{L^p} \| F (s) \|_{L^q} \le C  \| F (s) \|_{L^q}(t - s)^{\frac n {2p} - \frac{n+1}{2}} \quad\text{ for }\frac{1}{r}=\frac{1}{q}+\frac{1}{p}-1.
\end{equation*}
Thus for $q \in (1, n)$ we have
\begin{equation}
\label{eq:heat regularisation Lp}
    F \in C([0,T]; L^1) \cap C([\delta, T]; L^q) \implies u (\cdot; F) \in C([2 \delta, T]; L^r ), \qquad r < \frac{n q}{n-q}.
\end{equation}

Now we can obtain our first result of existence and uniqueness for \eqref{eq:ADE rho}, generalising the results of \cite{Canizo+Carrillo+Schonbek2012}, and fitting our current purpose.
\begin{theorem}[Local in time well-posedness]
    \label{thm:existence}
    Given $\rho_0 \in L^1_+ (\mathbb R^n)$ 
    and
    $\nabla W \in L^\infty$ there exists a unique solution $\rho(t)$ in $C([0,T]; L^1 (\Rd))$ for some $T > 0$ of \eqref{eq:ADE rho} in the sense that it satisfies \eqref{eq:ADE mild solution}. 
    The solution 
    has a maximal existence time $T^*$. If $T^* < \infty$ then
\begin{equation}\label{eq_blow_up}
    \lim_{t \to (T^*)^-} \|\rho(t) \|_{L^1}= +\infty.
\end{equation}
Furthermore, let $\rho$ and $\overline \rho$ solutions of \eqref{eq:ADE mild solution} corresponding to initial data $\rho_0$ and $\overline \rho_0$ respectively.
    We have that
    \begin{equation*}
     \sup_{t \in [0,T]} \| \rho_t  - \overline \rho_t \|_{L^1} \le C(T) \| \rho_0 - \overline \rho_0 \|_{L^1} .
\end{equation*}
\end{theorem}

\begin{proof} 
 We apply Banach's fixed-point theorem in $X = C([0,T]; Y)$, where $Y = \{ \rho \in L^1 (\Rd) : \| \rho \|_{L^1} \le \| \rho_0 \|_{L^1} + 1 \}$, 
to the solutions of $u_t - \Delta u = \diver F$ and $F = \rho \nabla W * \rho$. Hence, 
we define an operator $\mathcal F$ through the right-hand side of \eqref{eq:ADE mild solution}, i.e.
\begin{equation}
    \mathcal F[\rho] (t) = u(t; \rho \nabla W * \rho ). 
\end{equation}
We first point out that, by Young's convolution inequality
\begin{equation}
\label{eq:rho nabla W * rho respect to rho}
\begin{aligned}
    \| \rho (\nabla W * \rho)&  - \overline \rho (\nabla W * \overline \rho) \|_{L^1} \\
    &\le \| \rho (\nabla W * \rho)  - \rho (\nabla W * \overline \rho) \|_{L^1} + \| \rho (\nabla W * \overline \rho)  - \overline \rho (\nabla W * \overline \rho) \|_{L^1}\\
    &\le \|\rho\|_{L^1} \| \nabla W * (\rho  - \overline \rho) \|_{L^\infty} + \| \rho - \overline \rho\|_{L^1} \| \nabla (W * \overline \rho)\|_{L^\infty} \\
    &\le \|\rho\|_{L^1} \| \nabla W \|_{L^\infty}  \| \rho  - \overline \rho \|_{L^1} + \| \rho - \overline \rho\|_{L^1} \| \nabla  W \|_{L^\infty} *  \| \overline \rho\|_{L^1}\\
    &\le ( \| \rho \|_{L^1} + \| \overline \rho \|_{L^1} ) \|\nabla W\|_{L^\infty} \| \rho - \overline \rho \|_{L^1}.
\end{aligned}
\end{equation}
This means, on the one hand that it does not reduce the modulus of continuity in time of $\rho$, since
\begin{equation}
\label{eq:modulus of continuity of rho nabla W * rho}
\begin{aligned}
    \| \rho(t) (\nabla W * \rho(t))&  - \rho(s) (\nabla W * \rho(s)) \|_{L^1} \le ( \| \rho (t) \|_{L^1} + \|  \rho (s) \|_{L^1} ) \|\nabla W\|_{L^\infty} \| \rho(t) - \rho(s) \|_{L^1}.
\end{aligned}
\end{equation}
We check that $\mathcal F: X \to X$ by joining \eqref{eq:modulus of continuity of rho nabla W * rho} with \eqref{eq:modulus continuity heat} and \eqref{eq:heat L1 bound} for $T$ small enough.
Let us now show that $\mathcal F$ is contractive for $T$ small enough. Pick $\rho, \overline \rho \in X$. Due to \eqref{eq:rho nabla W * rho respect to rho} and \eqref{eq:heat continuous dependence F}
\begin{equation*}
    \| \mathcal F [\rho] - \mathcal F[\overline \rho] \|_X \le C T^{\frac 1 2} \| \rho \nabla W * \rho - \overline \rho \nabla W * \overline \rho \|_X \le C ( \| \rho \|_X + \| \overline \rho \|_X ) T^{\frac 1 2} \|\nabla W \|_{L^\infty} \| \rho - \overline \rho \|_X.
\end{equation*}
We can select
$T > 0$ small so that there is a contraction. 
Lastly, let us show the continuous dependence. With a similar argument as above we obtain that
\begin{equation*}
    \| \rho_t  - \overline \rho_t \|_{L^1} \le \| \rho_0 - \overline \rho_0 \|_{L^1} + C T ^{\frac 1 2 } \|\nabla W \|_{L^\infty} \sup_{s \in [0,T] } \|\rho_s - \overline \rho_s \|_{L^1}.
\end{equation*}
Hence, for $T$ small enough that $C T ^{\frac 1 2 } \|\nabla W \|_{L^\infty} < 1$,  
\begin{equation*}
    \sup_{t \in [0,T]} \| \rho_t  - \overline \rho_t \|_{L^1} \le \frac{ 1 }{ 1 - C \| \nabla W \|_{L^\infty} T^{\frac 1 2}} \| \rho_0 - \overline \rho_0 \|_{L^1} .
\end{equation*}
Since $C$ does not depend on $\rho$, this argument can be applied iteratively to deduce the result.
\end{proof} 
A similar argument provides continuous dependence on $\nabla W$.
The next theorem is our main result in this section. We will apply a bootstrap argument to show the solution $\rho$ in \Cref{thm:existence} is in $C((0,T^*);\mathcal{W}^{s,p}(\mathbb{R}^n))$ during its existence for any $s>0$ and $p\in [1,\infty)$, and in fact we have $T^*=\infty$, i.e. the solution is global in time. Once the regularity in space is shown, we immediately obtain the regularity of $\rho$ in time by passing it through the equation \eqref{eq:ADE rho}, thus $\rho$ is a classical solution of \eqref{eq:ADE rho}.

\begin{theorem}[Global in time solutions and instant regularisation]
    \label{thm:instant regularisation}
    Let $W \in \mathcal W^{1,\infty} (\mathbb R^n)$. Then the solution constructed in \Cref{thm:existence} is defined for all $T > 0$ and it satisfies
    \begin{equation*}
        \rho(t) \ge 0, \qquad \int_\Rd \rho(t) = \int_\Rd \rho_0,
    \end{equation*}
    and, 
    $\rho \in C ((0, T]; \mathcal W^{k,p} (\Rd))$, for any $k \in \mathbb N$ and $p \in [1,\infty]$. 
    Furthermore, $\rho$ is a classical solution defined for all $t > 0$. 
    In fact, if in addition $\rho_0 \in \mathcal W^{s,p} (\Rd)$, then $\rho \in C([0,T]; \mathcal W^{s,p} (\Rd))$ for any $s \ge 0$ and $p\in [1,\infty]$.
\end{theorem}

Before presenting the proof, let us first introduce some preliminaries. The proof of the regularity result is based on an iteration argument in fractional Sobolev spaces $\mathcal{W}^{s,p}$, whose definition and basic properties can be found in \Cref{sec:fractional Young-Sobolev inequality}. 
The reason to use fractional spaces is that our iterative scheme does not seem to be able to jump between $ \rho  \in C( (0,T^*) , L^\infty )$ and $ u( \cdot ; \rho  (\nabla W*\rho) ) \in C((0,T^*) , \mathcal W^{1,1})$, but we can gain fractional regularities to bridge the integer gap.

In each step of the iteration, assuming that $\rho\in C((0,T]; \mathcal{W}^{s,p})$ for certain $s\geq 0$, $p\in[1,\infty)$, we aim to use the formula \eqref{eq:Duhamel u} to upgrade the regularity to a higher order. This will be done by controlling the fractional Sobolev norm of $\nabla G_{t-s} * F(s)$, where $F(s)=\rho(s)(\nabla W*\rho(s))$. The following two key ingredients will be used in this estimate:

\begin{enumerate}
    \item
To obtain estimates on fractional Sobolev norms of a convolution, we need a Young's inequality between fractional Sobolev spaces. We could not locate such a result in the literature, so we provide a proof in \Cref{thm:fractional Young inequality}, which might be of independent interest.

\item
In order to control the fractional Sobolev norms of $F(s)=\rho(s)(\nabla W*\rho(s))$ itself, we need a product estimate in fractional Sobolev spaces. 
An estimate of this kind
was obtained by Brezis and Mironescu \cite{Brezis2001}: 
\begin{equation}\label{ineq_product}
    \|f g \|_{\mathcal W^{\theta s, p}} \le C ( \| f \|_{L^\infty} \| g \|_{\mathcal W^{\theta s,p}} + \| g \|_{L^r} \|f \|_{\mathcal W^{s,t}} ^\theta \|f \|_{L^\infty}^{1-\theta}),
\end{equation}
where $p,r,t \in (1,\infty)$, $s \in (0,\infty)$, $\theta \in (0,1) $ are such that 
$ \frac 1 r + \frac \theta t = \frac 1 p.$
However, a delicate issue is that we only assume $\nabla W \in L^\infty$ in this section, thus $\nabla W*\rho(s)$ can only belong to $L^\infty$-based spaces (such as $C^s=\mathcal{W}^{s,\infty}$).
In particular, it is impossible to show it belongs to $W^{s,p}$ for any $p<\infty$. For this reason, we could not apply \eqref{ineq_product} since it requires $p,t < \infty$. 

In the following lemma, we derive a product estimate for the fractional Sobolev norm of $fg$ where $f\in \mathcal{W}^{s,p}$ and $g\in C^s$. It can be seen as a minor generalisation of \eqref{ineq_product} with $t=\infty$, and we give a short direct proof.

\end{enumerate}

\begin{lemma} \label{lemma2.3}
Let $p \in [1,\infty)$, $s, \theta \in [0,1)$. If  $f \in C^s$ and $g \in W^{\theta s, p}$, then $f g \in W^{\theta s, p}$, and we have the estimate
\begin{equation}
\label{eq:Sobolev norm of a product}
    [f g]_{\mathcal W^{\theta s, p}} \le C(p,s,\theta) ( \| f \|_{L^\infty} [ g ]_{\mathcal W^{\theta s,p}} + \| g \|_{L^p} \| f \|_{C^{s}} )
\end{equation}
where $\| f \|_{C^{s}} = [f]_{C^s} + \| f \|_{L^\infty}$. 
\end{lemma}
\begin{proof}
    If $s=0$ or $\theta=0$, clearly $\|f g \|_{L^p} \le \|f \|_{L^\infty} \|g \|_{L^p}$. For $s,\theta\in(0,1)$, we write
    \begin{align*}
        &\int_\Rd \!\int_\Rd \! \frac{|f(x)g(x) - f(y) g(y)|^p}{|x-y|^{n+\theta s p}} \le C(p) \int_\Rd \! \int_\Rd \! \frac{|f(y)|^p|g(x) - g(y)|^p}{|x-y|^{n+\theta s p}} + C(p) \int_\Rd \! \int_\Rd \!\frac{|f(x) - f(y)|^p |g(x)|^p}{|x-y|^{n+\theta s p}} \\
        & \le C(p) \|f \|_{L^\infty}^p [g]_{W^{\theta s,p}}^p + C(p) \int_\Rd |g(x)|^p \left(\int_{\{|y-x| < 1\}}\!\! \frac{[f]_{C^s}^p|x-y|^{sp} }{|x-y|^{n+\theta s p}} \diff y + \int_{\{|y-x| \ge 1\}} \frac{2^p \| f \|_{L^\infty}^p }{|x-y|^{n+\theta s p}} \diff y\right) \diff x.
    \end{align*}
    Since $\int_{|y| < 1} |y|^{-n+(1-\theta)sp} \diff y , \int_{|y| \ge 1} |y|^{-n-\theta sp} \diff y < C(p,s,\theta)$ we conclude the result.
\end{proof}
We now have all the machinery needed for the proof of the main result of this section.
\begin{proof}[Proof of \Cref{thm:instant regularisation}] 
Iterating in \eqref{eq:heat regularisation Lp} and using Young's inequality, we get
$\rho \in C( (0,T^*) ; L^p (\Rd))$ for any $p \in [1,\infty)$. To recover higher regularity we pass through fractional Sobolev spaces.
We begin by proving some further regularity estimates for $u(t;F)$. Applying \Cref{thm:fractional Young inequality}, we have that
\begin{equation*}
    \| \nabla  G_{t-\tau} * F (s) \|_{\mathcal W^{\gamma,r}} \le  C \| \nabla  G_{t-\tau} \|_{\mathcal W^{\alpha,p}} \| F (s) \|_{\mathcal W^{\beta,q}} \le C  \| F (s) \|_{\mathcal W^{\beta,q}}(t - \tau)^{\frac n {2p} - \frac{n+1+\alpha}{2}}
\end{equation*}
where $\gamma = \alpha + \beta$. The time term is integrable if
$
    1 \le p < \frac{n}{n+\alpha-1} .
$
Hence, necessarily $\alpha < 1$, and we deduce that
\begin{equation}
\label{eq:heat regularisation fractional Sobolev}
\begin{aligned} 
    F \in C([0,T]; L^1 (\Rd)) &\cap C([\delta, T]; \mathcal W^{\beta,q}(\Rd)) \implies u (\cdot; F)  \in C([ \delta , T]; \mathcal W^{\gamma,r}(\Rd) ), \\
    & \text{where } \alpha < 1, \quad   1 \le p < \tfrac{n}{n+\alpha-1}, \quad  \tfrac 1 r + 1 = \tfrac 1 p + \tfrac 1 q, \quad \gamma = \alpha + \beta .
\end{aligned} 
\end{equation}
Applying \eqref{eq:heat regularisation fractional Sobolev} with $\beta = 0, p = 1, \alpha \in (0,1), q = r \in [1,\infty)$ we recover that $\rho \in C((0,T^*); W^{\alpha,q} (\Rd))$.

Let us reinterpret \eqref{eq:Sobolev norm of a product} for $f = \frac{\partial }{\partial x_i} (W*\rho)$ and $g = \rho$. For $s, \theta \in (0,1)$, applying %
\Cref{lemma2.3}, 
we have that
\begin{equation}
\label{eq:existence Wsp estimate 1}
\begin{aligned}
    \left[\rho \tfrac{\partial W}{\partial x_i} * \rho  \right]_{\mathcal W^{\theta s, p}} 
    &\le  
    C 
    \left (
        \left\| \tfrac{\partial W}{\partial x_i} *\rho   \right\|_{L^\infty} [ \rho ]_{\mathcal W^{\theta s,p}} 
        + \| \rho \|_{L^p} \|\tfrac{\partial W}{\partial x_i} *\rho \|_{C^s} \right ) \\
    &\le 
    C \left (
        \left\| \tfrac{\partial W}{\partial x_i}   \right\|_{L^\infty} [ \rho ]_{\mathcal W^{\theta s,p}} 
        + \| \rho \|_{L^p}  \| \rho \|_{\mathcal W^{s,1}}  \|\tfrac{\partial W}{\partial x_i}  \|_{L^\infty} \right ).
\end{aligned}
\end{equation}
Using the standard Young inequality
\begin{equation}
\label{eq:existence Wsp estimate 2}
    \left\|\rho \tfrac{\partial W}{\partial x_i} * \rho  \right\|_{L^p} \le  \left\|\rho \right\|_{L^p} \left\| \tfrac{\partial W}{\partial x_i} * \rho  \right\|_{L^\infty} \le \left\|\rho \right\|_{L^p}\left\|\rho \right\|_{L^1} \left\| \tfrac{\partial W}{\partial x_i}   \right\|_{L^\infty}.
\end{equation}
Then, we have that
\begin{equation*}
     F = \rho \nabla W * \rho \in C([\delta,T] , W^{ s, p  } (\Rd)), \qquad \forall s \in (0,1), p \in [1,\infty) .
\end{equation*}
This allows us to show, applying again \eqref{eq:heat regularisation fractional Sobolev}, that $\rho \in C((0,T^*); W^{2s,p} (\Rd))$ for $s \in (0,1)$ and $p \in [1,\infty)$. 
We can repeat the argument for $s \in (1,2)$ by noticing that $\frac{\partial}{\partial x_j} (\rho \tfrac{\partial W}{\partial x_i} * \rho) = \frac{\partial \rho }{\partial x_j} \tfrac{\partial W}{\partial x_i} * \rho + \rho \tfrac{\partial W}{\partial x_i} * \frac{ \partial \rho}{\partial x_j}$, and the reasoning above works in each element. Similar formulas hold for higher derivatives of $F$, and hence the argument can be extended to any $s > 0$.
Once we have space regularity, through \eqref{eq:ADE rho} time regularity follows. 

It remains to show that the solution is global in time. Towards this end, we will show that $\rho_-(\cdot,t) \equiv 0$ for all $t \in [0,T^*)$. For a smooth and convex function $j:\mathbb{R}\to\mathbb{R}$, we can write
\begin{align}
   \nonumber \frac{d}{dt} \int_{\mathbb R^n} j( \rho(x,t) ) \diff x &= - \int_{\mathbb R^n} j'' (\rho(t)) \Big( |\nabla \rho(t)|^2 + \rho \nabla \rho \cdot \nabla (W * \rho) \Big) \diff x \\
    &\le \int_{\mathbb R^n} \left( \int_0^{\rho(x,t)} j''(s) s \diff s \right) \Delta (W * \rho)(x) \diff x.\label{eq_positivity}
\end{align}
Let us approximate the convex (but non-smooth) function $j(s)=\max\{-s,0\}$ by a sequence of smooth convex functions $\{j_\ee\}_{\ee>0}$, where $j_\ee=j$ in $[-\ee,0]^c$ (so $j_\ee''\equiv 0$ in $[-\ee,0]^c$) and satisfies $0\leq j_\ee''\leq 2\ee^{-1}$ in $[-\ee,0]$.
Hence $J_\ee(s) := \int_0^s j_\ee''(\sigma) \sigma \diff \sigma$ satisfies $|J_\ee(s)|\leq |s|$ for all $0<\ee<1$, and 
$\lim_{\ee\to 0^+}J_\ee(s)=0$ for all $s$. 
Since $\Delta (W * \rho) \in L^\infty (\Rd)$ for $t > 0$, sending $\ee\to 0^+$ and applying the dominated convergence theorem to the right hand side of \eqref{eq_positivity} gives $\rho_-(\cdot,t)\equiv 0$ during its existence.
Hence, $\|\rho(t)\|_{L^1} = \| \rho_0 \|_{L^1}$ for all $t \in [0,T^*)$, and due to the blow-up criteria \eqref{eq_blow_up} we know there is no blow-up in finite time, that is, $T^* = +\infty$.

When $\rho_0 \in L^1 (\Rd) \cap \mathcal W^{s,p} (\Rd)$, we want to extend the regularity to $ \rho \in C([0,T]; L^1 (\Rd) \cap \mathcal W^{s,p} (\Rd))$. The first step is to notice that \eqref{eq:heat regularisation fractional Sobolev} works also for $\delta = 0$. Since \eqref{eq:existence Wsp estimate 1} and \eqref{eq:existence Wsp estimate 2} are point-wise in $t$, they hold up to $t = 0$. And thus the result is proven. 
\end{proof}

\section{Sharp decay of the free energy and the entropy}
\label{sec:decay free energy}
First, we give the sharp decay rate of the free energy functional in original variables $E(t)$ given by \eqref{def_E} for a bounded interaction potential $W$. From now on, we will always assume that the interaction potential $W$ is even without specifying it.

\begin{lemma}
    \label{thm:free energy decay}
    Assume $W \in \mathcal W ^{1,\infty} (\Rd)$,
    and $\rho_0\in L^1_+(\mathbb{R}^n)$ satisfy $\int_\Rd \rho_0 \diff x = 1$ and $E[\rho_0]<\infty$, as introduced in \eqref{def_E}. 
    Then there exists a constant $c > 0$ depending on $\|W\|_{L^\infty}$ and $n$, such that
    \begin{equation} \label{ineq_E}
	E[\rho(t)] \le - \frac n 2 \log \left ( ct+ e^{ - \frac 2 n E[\rho_0]}\right )\quad\text{ for all }t\geq 0.
\end{equation} 
\end{lemma}

\begin{proof}
For the length of the proof, let us denote 
$E(t) := E[\rho(t)]$.
Taking the time derivative of $E(t)$, we have
\begin{align*}
	 \frac{\diff E}{\diff t} &= -\int_{\mathbb R^n} \rho \left|  \nabla \frac{\delta E}{\delta u}   \right|^2 \diff x = -\int_{\mathbb R^n} \rho \left|  \nabla \left( \log \rho + W * \rho \right)   \right|^2 \diff x = -\int_{\mathbb R^n} \rho \left|  \nabla \log \left(  \rho e^{W * \rho}   \right)    \right|^2 \diff x.
\end{align*} 
If we define the auxiliary function $u(x,t)$ as
$
	u := \rho e ^{ W * \rho },
$
the above becomes 
\begin{equation}\label{dEdt2}	\frac{\diff E}{\diff t} = -\int_{\mathbb R^n} u e^{- W* \rho}  \left|  \nabla\left(  \log  u \right)   \right|^2 \diff x = -4 \int_{\mathbb R^n} e^{- W* \rho}  \left|  \nabla \sqrt u  \right|^2 \diff x,
\end{equation}
where the last identity follows from the fact that $u |\nabla \log u|^2 = 4 |\nabla \sqrt u|^2$.
For bounded $W$, we have $\| W * \rho (t) \|_{L^\infty} \le \| W \|_{L^\infty} \| \rho (t) \|_{L^1} \le \|W\|_{L^\infty}$, where we used that $\|\rho(t)\|_{L^1}=\|\rho_0\|_{L^1}=1$. 
Applying this to \eqref{dEdt2} yields
\begin{equation}\label{eq_dEdt} 
	 \frac{\diff E}{\diff t} \le - 4e^{-\| W \|_{L^\infty}} \int_{\mathbb R^n} |\nabla \sqrt u |^2 \diff x.
\end{equation}
In the rest of the proof, we aim to obtain a lower bound on the integral $\int_{\mathbb R^n} |\nabla \sqrt u |^2 \diff x$ in terms of $E$ itself. Recall that $E$ can be written as
\begin{equation*}
    E = \int_{\mathbb R^n} \log \left(  \rho e^{\frac { W*\rho} {2}}  \right)   \rho(x) \diff x = \frac 1 p \int_{\mathbb R^n} \log \left(  \left(  \rho e^{\frac { W*\rho} {2}}  \right)^p \right)  \rho(x) \diff x,
\end{equation*}
where $p>1$ will be determined momentarily. Applying Jensen's inequality gives
\begin{align}
	\nonumber E %
	& \nonumber 
	\le \frac 1 p \log \left( \int_{\mathbb R^n}   \left(  \rho e^{\frac { W*\rho} {2}}  \right)^p  \rho \diff x \right)   %
	= \frac 1 p \log \left( \int_{\mathbb R^n}    u^{p+1}  e^{ (-\frac p 2 -1)  { W*\rho} }  \diff x \right)  %
	\\
	& 
	\le \frac 1 p \log \left( e^{ (1 + \frac{p} 2)\|W\| _{L^\infty}} \int_{\mathbb R^n}    u^{p+1}   \diff x \right). %
	\label{eq_E_final}
\end{align}
From now on, let us fix $p:=\frac{2}{n}$. For such $p$, the Gagliardo-Nirenberg inequality gives that
\begin{equation*}
	\int_{\mathbb R^n} u^{p+1} \diff x \le C(n)\left( \int_{\mathbb R^n} |\nabla \sqrt u|^2 \diff x \right) \left( \int_{\mathbb R^n} u \diff x\right)^{\frac 2 n} \le C(n) e^{\frac 2 n \| W \|_{L^\infty}} \int_{\mathbb R^n} |\nabla \sqrt u|^2 \diff x.
\end{equation*}
Combining this with \eqref{eq_dEdt} and \eqref{eq_E_final}, we have
\begin{align*}
	E& \le \frac n 2 \log \left( C(n) e^{(1+\frac 3 n) \| W \|_{L^\infty}} \int_{\mathbb R^n} |\nabla \sqrt u|^2    \diff x \right) %
	\le \frac n 2 \log \left(  - \frac 14 C(n) e^{(2 + \frac 3 n) \| W \|_{L^\infty}}\frac{\diff E}{\diff t}  \right).
\end{align*} 
This means
\begin{equation*}
	\frac{\diff E}{\diff t}  \le - c(n,\|W\|_{L^\infty}) e^{ \frac 2 n E(t)},
\end{equation*}
where $c(n,\|W\|_{L^\infty})=4 C(n)^{-1} e^{-(2+\frac{3}{n})\|W\|_{L^\infty}}$. Solving this differential inequality yields the inequality \eqref{ineq_E}, finishing the proof.
\end{proof}

We now focus on using these estimates to obtain {uniform-in-time bounds for the rescaled equation} \eqref{eq:Fokker-Plack rho scaled}.
Following \cite{Carrillo+Toscani2000}, we perform a time-dependent rescaling with the new time and spatial variables being 
\begin{equation}\label{def_tau}
\tau = \log \lambda(t), \quad y = \lambda^{-1}(t)x,
\end{equation}
where $\lambda(t) = \sqrt{2t+1}$.
Let the rescaled density $\rhos(\tau,y)$ be related to $\rho(t,x)$ by
\begin{equation}\label{def_tilderho}
	\rhos (\tau,y) = \lambda(t)^n \rho (t, x),
\end{equation}
or, equivalently, 
\[\rhos(\tau,y)=e^{n\tau} \rho\left(\frac{e^{2\tau}-1}{2}, e^\tau y\right).
\]
Note that $\rho(0,\cdot)=\rhos(0,\cdot)$ and the $L^1$ norm of $\rhos(\tau,\cdot)$ is preserved under the rescaling. In addition, if $\rho(t,x)$ satisfies the heat equation $\partial_t \rho = \Delta_x \rho$, it is well-known (see \cite{Carrillo+Toscani2000} for example) that $\tilde\rho(\tau,y)$ satisfies the Fokker--Planck equation $\partial_\tau \rhos = \Delta_y\rhos + \nabla_y\cdot(\rhos y).$

Next let us derive the equation satisfied by $\rhos$ when $\rho$ solves \eqref{eq:ADE rho}. Compared to the heat equation, $\partial_\tau \rhos$ has an additional term $e^{(n+2)\tau} \nabla_x \cdot (\rho\nabla_x(W*\rho))(t,x)$, and it suffices to express it in terms of the new variables $\tau,y$ as well as $\rhos$.
Using the definition of $\tau,y$ and $\rhos$, the convolution $(W*\rho)(t,x)$ can be expressed as
\begin{equation}\label{eq_convolution_new}
\begin{split}
(W*\rho)(t,x) &= \int_{\mathbb{R}^n} W(x-x') \rho(t,x') \diff x'= \int_{\mathbb{R}^n} W\left(\lambda(t)(y-y')\right) \rho(t,x') \lambda^n(t) \diff y'\\
&= \int_{\mathbb{R}^n} W\left(e^\tau (y-y')\right) \rhos(\tau,y') \diff y'  =:(\Ws*\rhos)(\tau,y),
\end{split}
\end{equation}
using the change of variables $y':=\lambda^{-1}(t)x'$ and \eqref{def_tilderho},
where $\Ws(\tau,y) := W(e^\tau y)$. As a result, the additional term in $\partial_\tau \rhos$ can be written as
\[
e^{(n+2)\tau} \nabla_x \cdot (\rho\nabla_x(W*\rho))(t,x) = \nabla_y\cdot (\rhos \nabla_y(\Ws*\rhos))(\tau,y),
\]
where we used that $\nabla_y=e^\tau \nabla_x$, as well as \eqref{def_tilderho} and \eqref{eq_convolution_new}. Finally this leads to the equation for $\rhos$ in rescaled variables:
\begin{equation*}
	\frac{\partial \rhos}{\partial \tau}  =  \Delta_y \rhos + \nabla_y\cdot (y \rhos ) +  \nabla_y \cdot ( \rhos \nabla_y  (\Ws * \rhos)  ).
\end{equation*}

\begin{remark}
	In the rescaled variables, even though $\Ws(\tau,\cdot)=W(e^\tau\cdot)$ is $\tau$-dependent, its $L^\infty$ norm remains uniformly bounded as long as $W\in L^\infty$, and one can easily check that \[\|\Ws(\tau,\cdot)\|_{L^\infty} = \|W\|_{L^\infty} \quad\text{ for all }\tau\geq 0.
	\]However, the $\mathcal W^{m,q}$ norm of $\Ws(\tau,\cdot)$ can be exponentially growing/decaying in $\tau$, depending on the values of $m$ and $q$. More precisely, for any multi-index $\alpha$ and $q\geq 1$, we have
	\begin{equation*}
		\int_{\mathbb R^n} |D^\alpha \Ws (\tau, y) |^q \diff y = e^{(|\alpha|q -n ) \tau} \int_{\mathbb R^n} |D^\alpha W (x) |^q \diff x,
	\end{equation*}
	which leads to
	\begin{equation}
	\label{eq:scaling classical Sobolev}
		\| D^\alpha \Ws(\tau,\cdot) \|_{L^q} = e^{(|\alpha| - \frac{n}q) \tau}\| D^\alpha W \|_{L^q}.
	\end{equation}
	As a result, if $W \in  \mathcal W^{m,q}$ with $m \le  n / q$,
	then $\|\Ws(\tau,\cdot)\|_{\mathcal W^{m,q}}$ is uniformly bounded above for all $\tau\geq 0$. Furthermore, if $m <  n  / q$ then $\|\Ws(\tau,\cdot)\|_{\mathcal W^{m,q}}$ decays to zero as $\tau\to\infty$. The same kind of estimates holds for fractional Sobolev norms (see \Cref{sec:scaling of fractional Sobolev norm}).
\end{remark}

The next step is to establish uniform-in-time bounds for the free energy, the second moment and, as a consequence, the entropy in rescaled variables.
Throughout the rest of this paper, we will focus on the analysis of the rescaled equation \eqref{eq:Fokker-Plack rho scaled}. For notational simplicity, we will suppress the $y$ subscript from $\nabla_y$ and $\Delta_y$. 
Also, all the time and spatial variables below related to $\rhos$ will be the rescaled variables, unless specified otherwise. For example, ``taking the time derivative'' stands for taking the $\tau$-derivative; and when $x$ appears below in $\rhos(\tau,x)$, it will stand for the rescaled spatial variable rather than the original one.

Let us point out that one of the main difficulties to study the rescaled equation \eqref{eq:Fokker-Plack rho scaled} is the lack of a monotone-decreasing free energy functional. If $\Ws$ were known to be independent of $\tau$, it is well-known that there would be a natural free energy functional $\widetilde{F}(\tau)$ associated to \eqref{eq:Fokker-Plack rho scaled}, given by
\begin{equation}\label{def_F}
    \widetilde{F}(\tau) := \int_{\mathbb{R}^n} \left( \rhos \log \rhos + \rhos \frac{|y|^2}2 +\frac{1}{2} \rhos (\Ws(\tau, \cdot)*\rhos) \right) \diff y.
\end{equation}
But, since $\Ws(\tau,\cdot)=W(e^\tau \cdot)$ is $\tau$-dependent, $\widetilde{F}(\tau)$ is not necessarily decreasing in time. In fact, taking the time derivative of $\widetilde{F}(\tau)$ yields
\[
\frac{\diff}{\diff \tau}\widetilde{F}(\tau)= -\int_{\mathbb{R}^n} \rhos \left|\nabla\left(\log\rhos + \frac{1}{2}|y|^2+\Ws*\rhos\right)\right|^2 \diff y + \frac{1}{2} \int_{\mathbb{R}^n} \rhos \left (\frac{ \partial \Ws }{\partial \tau}* \rhos\right ) \diff y,
\]
where $\frac{ \partial \Ws }{\partial \tau} (\tau, y) = \frac{ \partial }{\partial \tau} [ W(e^\tau y) ] = e^\tau y \cdot \nabla W(e^\tau y) = y\cdot \nabla  \Ws(\tau, y)$. Plugging this into the above yields
\begin{equation}\label{dFdt}
\frac{\diff }{\diff \tau }\widetilde{F} (\tau)= -\int_{\mathbb{R}^n} \rhos \left|\nabla\left(\log \rhos + \frac{|y|^2}{2}+ \Ws *\rhos\right)\right|^2 \diff y  + \frac{1}{2}\iint_{\mathbb{R}^n\times \mathbb{R}^n} \rhos(y)\rhos(z) (y-z)\cdot \nabla \Ws(\tau, y-z) \diff y \diff z,
\end{equation}
where the right hand side is not necessarily negative due to the additional double integral. 

Instead of looking for a monotone free energy for the rescaled equation, let us consider a new free energy functional
\begin{equation}\label{def_Es}
	\Es(\tau) := \int_{\mathbb R^n} \left(  \rhos \log \rhos +  \frac 1 2 \rhos (\Ws(\tau,\cdot) * \rhos ) \right) \diff y.
\end{equation}
Even though this functional is not monotone in $\tau$, as we will show below, it has a natural relation with the free energy $E(t) = E[\rho(t)]$
defined in \eqref{def_E} in the original variable, and the sharp rate of decay of $E(t)$ that we established in \Cref{thm:free energy decay} implies a uniform-in-$\tau$ bound of $\Es(\tau)$.

\begin{lemma}
    Assume $W \in \mathcal W^{1,\infty} (\Rd)$,
    and $\rho_0\in L^1_+(\mathbb{R}^n)$ satisfy $\int_\Rd \rho_0 \diff x = 1$ and $E[\rho_0]<\infty$. The energy functionals $E(t)$ in \eqref{def_E} and $\Es(\tau)$ in \eqref{def_Es} satisfy that
    \begin{equation}\label{E_Es}
        E(t)=\Es(\tau)-n\tau \quad\text{ for all }\tau\geq 0,
    \end{equation}
    where $t$ and $\tau$ are related by \eqref{def_tau}. As a consequence, \Cref{thm:free energy decay} implies that
    \begin{equation}\label{ineq_Es}
    \Es(\tau) \leq C(\|W\|_{L^\infty}, n, E[\rho_0])\quad\text{ for all }\tau\geq 0.
    \end{equation}
\end{lemma}

\begin{proof}
Let us write the original energy $E(t)= \int_{\mathbb R^n} \left(  \rho \log \rho +  \frac 1 2 \rho (W * \rho ) \right) \diff x$ in terms of $\rhos$. For the entropy term, using \eqref{def_tau} and \eqref{def_tilderho} we have
\begin{align*}
	 \int_{\mathbb R^n}   \rho(x,t) \log \rho(x,t) \diff x &= \int_{\mathbb R^n}  \lambda^{-n}(t) \rhos(\tau,y) \log( \lambda^{-n}(t) \rhos(\tau,y)) \lambda^n(t) \diff y\\
	 &= \int_{\mathbb R^n} \rhos(\tau,y) \log \rhos(\tau,y) \diff y - n \tau,
\end{align*}
where in the last step we used that $\tau=\log\lambda(t) $ as well as $\int_{\mathbb{R}^n} \rhos(\tau,y)\diff y=1$. As for the interaction energy, using \eqref{def_tau} and \eqref{def_tilderho}  together with \eqref{eq_convolution_new}, we have
\begin{align*}
\int_{\mathbb{R}^n}\rho(t,x)(W*\rho)(t,x)\diff x 
&=\int_{\mathbb{R}^n}\lambda^{-n}(t)\rhos(\tau,y)(\Ws*\rhos)(\tau,y)\lambda^n(t)\diff y
= \int_{\mathbb{R}^n}\rhos(\tau,y)(\Ws*\rhos)(\tau,y)\diff y.
\end{align*}
Combining the above two identities together yields \eqref{E_Es}. Using \eqref{E_Es} and the inequality \eqref{ineq_E} for $E(t)$ we have, recalling $\tau=\log\sqrt{2t+1}$,
\[
\begin{split}
    \Es(\tau) = E(t)+n \log\sqrt{2t+1}
    &\leq \frac{n}{2}\log\left(\frac{2t+1}{c(\|W\|_{L^\infty},n)t+e^{-\frac{2}{n}E[\rho_0]}}\right)
    \\
    &
    \leq C(\|W\|_{L^\infty},n,E[\rho_0])\quad\text{ for all }\tau\geq 0,
\end{split}
\]
where the last inequality follows from the fact that for all $t\geq 0$, the fraction in the second line is uniformly bounded above by some constant only depending on $\|W\|_{L^\infty},n$ and $E[\rho_0]$. This finishes the proof of \eqref{ineq_Es}.\end{proof}

\begin{remark}
For $W\in L^\infty$, since the interaction energy satisfies \[
\frac{1}{2}\left|\int_\Rd\rhos(\Ws*\rhos)\diff y\right|\leq \frac{1}{2}\|\Ws\|_{L^\infty}\|\rhos\|_{L^1}^2 = \frac{1}{2}\|W\|_{L^\infty}\quad\text{ for all }\tau\geq 0,
\]
the bound \eqref{ineq_Es} on $\Es$  immediately implies that the entropy for the rescaled equation is uniformly bounded above:
\begin{equation}
	\label{eq:entropy bound}
	\int_{\mathbb R^n}  \rhos (\tau,y) \log \rhos ( \tau,y) \diff y \le \Es(\tau) + \frac{1}{2}\left|\int_\Rd \rhos(\Ws*\rhos)\diff y\right| \leq C(\|W\|_{L^\infty}, n, E[\rho_0])\quad\text{ for all }\tau\geq 0.
\end{equation}
\end{remark}

Let us now prove a {uniform-in-time bound of the second moment} in rescaled variables.
In \eqref{eq:entropy bound}, we have obtained a uniform-in-time bound of the entropy $\int\rhos(\tau,y)\log\rhos(\tau,y) dy$. In order to upgrade it into a uniform-in-time $L \log L$ norm of $\rhos$, we need some uniform-in-time tightness of $\rhos(\tau,\cdot)$. Our next goal is to obtain a uniform-in-time bound of the second moment of $\rhos(\tau,\cdot)$, given by
\begin{equation*}
	\mathcal N_2(\tau):= \mathcal N_2 [\rhos(\tau)] = \int_{\mathbb R^n} |y|^2 \rhos (\tau,y) \diff y.
\end{equation*}

A natural starting point is to track the evolution of $\mathcal N_2(\tau)$ in time. Taking its time derivative and integrating by parts in space, we deduce that
\begin{equation}\label{dN2dt}
\begin{split}
\frac{\diff}{\diff \tau} \mathcal{N}_2(\tau) &= -\int_{\mathbb{R}^n} 2y\cdot(\nabla\rhos + \rhos y + \rhos \nabla(\Ws * \rhos)) \diff y\\
&= 2n - 2\mathcal{N}_2(\tau) -  2\iint_{\mathbb{R}^n\times \mathbb{R}^n} \rhos(y)\rhos(z) y \cdot \nabla \Ws(y-z) \diff y \diff z\\
&= 2n - 2\mathcal{N}_2(\tau) -  \iint_{\mathbb{R}^n\times \mathbb{R}^n} \rhos(y)\rhos(z) (y-z) \cdot \nabla \Ws(y-z) \diff y \diff z,
\end{split}
\end{equation}
where the last identity is obtained by exchanging $y$ and $z$ in the integrand and taking average with the original integral. 

Note that if $W$ is attractive (i.e. $W$ is radially increasing), we have that $x\cdot\nabla W(x)\geq 0$ for all $x$, and the same is true for the rescaled potential $\Ws(\tau,\cdot)$. This leads to the differential inequality $\frac{\diff}{\diff \tau}\mathcal{N}_2(\tau) \leq 2n-2\mathcal{N}_2(\tau)$, which yields a uniform-in-time upper bound of $\mathcal{N}_2(\tau)$. However, this argument fails for a general bounded potential $W$ that is not necessarily attractive.

To overcome this difficulty, instead of tracking the time derivative of $\mathcal{N}_2(\tau)$ itself, the idea is to take a linear combination with the functional $\widetilde{F}(\tau)$ in \eqref{def_F}, so that the double integral involving $\nabla\Ws$ will be cancelled in their time derivatives. The result is as follows:

\begin{theorem}\label{thm_m2}
    Let $W \in \mathcal W^{1,\infty}(\mathbb{R}^n)$, and assume $ \rho_0\in L^1_+(\mathbb{R}^n)$ with $\int_{\mathbb{R}^n}\rho_0 \diff x = 1$, $E[\rho_0]<\infty$, and $\mathcal{N}_2[\rho_0]<\infty$. Then we have
    \begin{equation}
    \label{eq:moment 2 bound}
     \mathcal N_2(\tau) = \mathcal N_2[\rhos(\tau)] \le C(\mathcal{N}_2[\rho_0], \|W\|_{L^\infty}, n, E[\rho_0]).
\end{equation}
\end{theorem}

\begin{proof} 
Let $\widetilde{F}(\tau)$ be defined as in \eqref{def_F}, and recall that its time derivative is given by \eqref{dFdt}. Comparing \eqref{dFdt} with \eqref{dN2dt}, we observe that for the linear combination $\widetilde{F}(\tau)+\frac{1}{2}\mathcal{N}_2(\tau)$, the double-integrals in their time derivative exactly cancel each other. More precisely, we have
\begin{equation*}%
\begin{split}
\frac{\diff}{\diff\tau} \left(\widetilde{F}(\tau)+\frac{1}{2}\mathcal{N}_2(\tau)\right) &=-\int_{\mathbb{R}^n} \rhos \left|\nabla\left(\log \rhos + \frac{|y|^2}{2}+ \Ws *\rhos\right)\right|^2 \diff y + n - \mathcal{N}_2(\tau)
\leq  n - \mathcal{N}_2(\tau).
\end{split}
\end{equation*}
Recall that $\widetilde{F}(\tau)=\Es(\tau)+\frac{1}{2}\mathcal{N}_2(\tau)$, and $\Es(\tau)$ has a uniform-in-time upper bound due to \eqref{ineq_Es}. Therefore
\begin{align*}
 \frac{\diff }{\diff \tau } (\widetilde E(\tau)  + \mathcal{N}_2(\tau) ) &\le  n - \mathcal{N}_2(\tau) = n + \widetilde E(\tau) - \left(\widetilde E (\tau) + \mathcal{N}_2(\tau)\right)  
 \le n + \sup_{\sigma \ge 0} \widetilde E(\sigma) -  \left(\widetilde E (\tau) + \mathcal{N}_2(\tau)\right).
\end{align*} 
Multiplying by $e^{\tau}$ and integrating, we have that
\begin{equation}\label{eq_upper_bound_sum}
    \widetilde E (\tau) + \mathcal{N}_2(\tau) \le e^{-\tau}(\widetilde E (0) + \mathcal{N}_2(0)) + (1-e^{-\tau})\left( n + \sup_{\sigma \ge 0} \widetilde E(\sigma)\right) \leq C(\mathcal{N}_2[\rho_0], \|W\|_{L^\infty}, n, E[\rho_0])
\end{equation}
for all $\tau\geq 0$, where in the second inequality we used \eqref{ineq_Es} and the fact that $\Es(0)=E[\rho_0]$.

Note that this inequality does not yield an upper bound for $\mathcal{N}_2(\tau)$ yet, since we do not know whether $\Es$ is bounded below. Now we write $ \Es (\tau) + \mathcal{N}_2(\tau)$ back into $\widetilde{F}(\tau) + \frac{1}{2}\mathcal{N}_2(\tau)$, and use the crucial fact that
$\widetilde{F}$ is bounded below by a constant only depending on $n$ and $\|W\|_{L^\infty}$, since
\begin{align*}
    \widetilde{F} (\tau) &= \int_\Rd \left( \rhos \log \rhos + \rhos \frac{|y|^2}2 +\frac{1}{2} \rhos (\Ws*\rhos) \right) \diff y 
    \ge \int_\Rd \left( \rhos \log \rhos + \rhos \frac{|y|^2}2 \right) \diff y - \frac{1}{2} \|W \|_{L^\infty} 
    \\
    &
    \ge - C(n) -  \frac{1}{2}\|W \|_{L^\infty} ,
\end{align*}
where the integral in the second line is the free energy of the Fokker-Planck equation, which is minimised at the Gaussian profile. Combining the lower bound of $\widetilde{F} (\tau)$ with the upper bound of $\widetilde{F} (\tau) + \frac{1}{2}\mathcal{N}_2(\tau)$ in \eqref{eq_upper_bound_sum}, we finish the proof of \eqref{eq:moment 2 bound}.
\end{proof}

Finally, we obtain a uniform bound in $L \log L$ in rescaled variables.
Joining \eqref{eq:entropy bound} and \eqref{eq:moment 2 bound} we recover a uniform-in-time bound of $\int_\Rd  \rhos  (\tau) |\log \rhos (\tau) |$ by classical techniques \cite{BD95,BCC12} (we give a general result in \Cref{lem:rho log rho + second moment implies rho | log rho |} which may be of independent interest).
\begin{corollary}\label{cor_entropy}
	If $W \in \mathcal W^{1,\infty} (\Rd)$,
	then the solution of \eqref{eq:Fokker-Plack rho scaled} satisfies
	\begin{equation}
		\label{eq:entropy rescaled finite}
		\int_\Rd  \rhos  (\tau) |\log \rhos (\tau) |  \le C (n, \| W\|_{L^\infty},E [\rho_0], \mathcal N_2 [\rho_0]).
	\end{equation}
\end{corollary}

\section{Propagation of regularity for the rescaled density \texorpdfstring{$\rhos$}{rho} }
\label{sec:prop regularity}

\subsection{Uniform-in-time bounds of \texorpdfstring{$L^2$}{L2} and \texorpdfstring{$H^1$}{H1} norms}

\begin{theorem}
    \label{thm:propagation of L2 and H1 regularity}
    Let $n \ge 1$,  $
    W \in \mathcal W^{1,\infty}(\mathbb{R}^n)$ and
    $\nabla W \in L^n (\Rd)$.
    Assume $ \rho_0\in L^1_+(\mathbb{R}^n)$ with
    $\int_{\mathbb{R}^n}\rho_0 \diff x = 1$.
    Let $\rho(x,t)$ be the solution constructed in \Cref{thm:existence} with initial data $\rho_0$. Then the rescaled density $\rhos(\tau,y)$ defined in \eqref{def_tau}--\eqref{def_tilderho} satisfies the following:
    \begin{enumerate}
        \item If $\rho_0 \in L^2 (\mathbb R^n)$ with $\mathcal{N}_2[\rho_0]<\infty$, then $\rhos \in L^\infty(0,\infty; L^2 (\mathbb R^n))$ with the estimate
        \begin{equation}\label{eq_L2_uniform}
            \| \rhos (\tau) \|_{L^2} \le C(n, \|W\|_{L^\infty},\|\nabla W\|_{L^n}, \|\rho_0 \|_{L^2}, \mathcal N_2 [\rho_0] ) \quad\text{ for all }\tau\geq 0.
        \end{equation}
        \item If $\rho_0 \in H^1 (\mathbb R^n)$ with $\mathcal{N}_2[\rho_0]<\infty$, then $\rhos \in L^\infty(0,\infty; H^1 (\mathbb R^n))$ with the estimate
        \begin{equation}\label{eq_H1_uniform}
            \| \rhos (\tau) \|_{H^1} \le C(n,\|W\|_{L^\infty}, \|\nabla W\|_{L^n},\|\rho_0 \|_{H^1}, \mathcal N_2 [\rho_0]) \quad\text{ for all }\tau\geq 0.
        \end{equation}
    \end{enumerate}
    
\end{theorem}
With this result, by compactness we can easily prove  that
\begin{corollary}
    If $W \in L^\infty (\Rd)$ and $\nabla W \in L^n$ for $\rho_0 \in L^2 (\mathbb R^n)$ (resp. $H^1 (\mathbb R^n)$) with  $\mathcal{N}_2[\rho_0]<\infty$, there exists a mild solution of \eqref{eq:ADE rho} that satisfies the estimates above.
\end{corollary}

\begin{proof}[Proof of \Cref{thm:propagation of L2 and H1 regularity}]
By the instant regularisation result in \Cref{thm:instant regularisation} and the relation between $\rho$ and $\rhos$, we know that $\rhos \in C^1 ((0,T]; H^2 (\mathbb R^n))$, even if 
we do not have an estimate of $\|\rho(t)\|_{H^2}$.
To obtain uniform-in-time estimates on the $L^2$ norm of $\rhos$, we will track the $L^2$ norm evolution of $\rhos_k := (\rhos - k)_+$, where $k>1$ is a constant to be determined later. To begin with, we list some properties of $\rhos_k$.

\noindent \textbf{Step 1. Relation between $\rhos_k = (\rhos - k)_+$ and $\rhos$.} Due to \eqref{eq:entropy rescaled finite} we have, 
for any $k > 1$ and $\tau\geq 0$, that
\begin{equation}
	\label{eq:equiintegrability from entropy}
\|\rhos_k(\tau)\|_{L^1} \le \int_{\{\rhos > k\}} \rhos \le \frac{1}{\log k} \int_{\{ \rhos > k \}} \rhos \log \rhos \le  \frac{1}{\log k} \int_{ \mathbb R^n } \rhos | \log \rhos | \le \frac{C_0}{\log k}. 
\end{equation} 
where $C_0 = C(n, \| W\|_{L^\infty},E [\rho_0], \mathcal N_2 [\rho_0])$. Note that since $\|\rho_0\|_{L^1}=1$, $E[\rho_0]$ can be bounded above using $\|\rho_0\|_{L^2}$ and $\|W\|_{L^\infty}$.

Next we state an inequality relating the $L^p$ norm of $\rhos_k(\tau,\cdot)$ and $\rhos(\tau,\cdot)$ (below the $\tau$ dependence is compressed for notational simplicity). Since $\rhos = \rhos_k + \min\{\rhos,k\}$, combining the triangle inequality on the $L^p$ norm with H\"older's inequality gives
\begin{equation}\label{eq:relation rhos and rhosk p < infty}
\|\rhos\|_{L^p} \leq \|\rhos_k\|_{L^p} + \|\min\{\rhos,k\}\|_{L^p} \leq  \|\rhos_k\|_{L^p} + k^{\frac{p-1}{p}} \quad\text{ for } p\in [1,\infty),
\end{equation}
 where the second inequality follows from $\|\min\{\rhos,k\}\|_{L^\infty}\leq k$ and $\|\min\{\rhos,k\}\|_{L^1}\leq \|\rho_0\|_{L^1} = 1$. 
For $p = \infty$ we simply use the fact that
\begin{equation}
    \label{eq:relation rhos and rhosk p = infty}
    \|\rhos\|_{L^\infty} \leq \|\rhos_k\|_{L^\infty} +k.
\end{equation}
Hence, if $\rhos_k \in L^\infty(0,T; L^2 (\Rd))$, then so is $\rhos$.

\noindent \textbf{Step 2. Evolution of $L^2$ norm of $\rhos_k$.}
We compute
\begin{align}
    \frac 1 2 \frac{\diff}{\diff \tau} \int_{\mathbb R^n} \rhos_k^2 &= - \int_{\mathbb R^n} \nabla \rhos_k \cdot ( \nabla \rhos + \rhos \nabla (\Ws * \rhos) + \rhos y ) \diff y \nonumber \\
    &= -\| \nabla \rhos_k \|_{L^2}^2 - \underbrace{\int_\Rd \nabla \rhos_k \cdot ( \rhos \nabla (\Ws * \rhos) ) \diff y}_{=:J_1} - \underbrace{\int_\Rd \rhos \nabla \rhos_k \cdot y \diff y}_{=:J_2}.\label{eq_dt_L2}
\end{align}
We first deal with the more complicated term $J_1$. We have
\begin{align*}
    |J_1| %
    & \le \|\nabla \rhos_k \|_{L^2} \|\rhos \|_{L^2} \|\nabla (\Ws * \rhos)\|_{L^\infty} 
    \le \|\nabla \rhos_k \|_{L^2} \| \rhos \|_{L^2} \|\nabla W \|_{L^{n}} \| \rhos\|_{L^{\frac {n}{n-1}}},
\end{align*}
where in the second step we used that $\|\nabla \Ws \|_{L^{n}} = \|\nabla W \|_{L^{n}}$, which is due to \eqref{eq:scaling classical Sobolev}. Note that the above computation holds for all $n\geq 1$,
where for $n = 1$ we use the notation $\frac{n}{n-1} = \infty$. 
Applying \eqref{eq:relation rhos and rhosk p < infty} (or \eqref{eq:relation rhos and rhosk p = infty} for  if $n = 1$) and using $k\geq1$, we recover
\begin{equation}\label{estimate_J1}
    |J_1|  \le \|\nabla \rhos_k \|_{L^2}  \|\nabla W \|_{L^{n}} \left( \| \rhos_k \|_{L^2} + k \right) \left( \| \rhos_k\|_{L^{\frac {n}{n-1}}} + k\right).
\end{equation}
The Gagliardo-Nirenberg inequality yields
\begin{align}
    \| \rhos_k \|_{L^{2}} &\le C(n) \| \nabla \rhos_k \|_{L^2}^{\frac n{n+2}} \| \rhos_k \|^{\frac{2}{n+2}}_{L^1} \label{eq:GN H1 and L1 to L2}\\
    \| \rhos_k \|_{L^{\frac {n}{n-1}}} &\le C(n) \| \nabla \rhos_k \|_{L^2}^{\frac 2{n+2}} \| \rhos_k \|^{\frac{n}{n+2}}_{L^1},\nonumber 
\end{align}
and plugging these two inequalities into \eqref{estimate_J1} gives \begin{align*}
    |J_1| &\le C(n) \|\nabla W\|_{L^n}  \| \rhos_k \|_{L^1} \|\nabla \rhos_k \|_{L^2}^2 \\
    & \quad + C(n,k) \|\nabla W\|_{L^n} \Big(  \| \rhos_k \|_{L^1}^{\frac 2 {n+2}} \|\nabla \rhos_k \|_{L^2}^{1 + \frac{n}{n+2}} + \| \rhos_k \|_{L^1}^{\frac n {n+2}} \|\nabla \rhos_k \|_{L^2}^{1 + \frac{2}{n+2}} + \|\nabla\rhos_k \|_{L^2} \Big).
\end{align*}
By applying Young's  inequality for products to each element in the second term we recover that for any $ 0 < \delta < 1 $,
\begin{equation*}
    |J_1| \le C(n) \|\nabla W\|_{L^n} \Big(\| \rhos_k \|_{L^1} + \| \rhos_k \|_{L^1}^{\frac{2}{n+1}} + \| \rhos_k \|_{L^1}^{\frac{2 n}{n+4}} + \delta \Big)  \| \nabla \rhos_k \|_{L^2}^2 + C(n,k)\| \nabla W \|_{L^n}\delta^{-1}. 
\end{equation*}
We now deal with the term $J_2$. This can be computed explicitly as
\begin{align*}
    J_2 &= \int_{\{\rhos > k\}} \rhos \nabla \rhos_k \cdot y \diff y = \int_{\{\rhos > k\}} (\rhos_k + k) \nabla \rhos_k \cdot  y \diff y 
    = \frac 1 2 \int_{\Rd} \nabla (\rhos_k^2) \cdot y \diff y + k \int_\Rd \nabla \rhos_k \cdot y \diff y \\
    &=-\frac n 2 \int_{\Rd}  \rhos_k^2 \diff y - nk \int_\Rd  \rhos_k  \diff y.
\end{align*}
Using \eqref{eq:GN H1 and L1 to L2} as well as the fact that $\|\rhos_k\|_{L^1}\leq 1$, we have
\begin{align*}
    |J_2| &\le C(n) \|\nabla \rhos_k \|_{L^2}^{\frac{2n}{n+2}} \|\rhos_k \|_{L^1}^{\frac{4}{n+2}} + nk 
    \le C(n) \|\nabla \rhos_k \|_{L^2}^{2} \| \rhos_k \|_{L^1}^{\frac{4}{n}}  + C(n,k).
\end{align*}
Plugging the $J_1$ and $J_2$ estimates into \eqref{eq_dt_L2}, we have that for any $ 0 < \delta < 1$ and $k\geq1$,
\begin{align*}
    \frac{\diff}{\diff \tau} \int_{\mathbb R^n} \rhos_k^2 &\le - \Bigg(2 - C(n, \|\nabla W\|_{L^n}) \Big(\| \rhos_k \|_{L^1} + \| \rhos_k \|_{L^1}^{\frac{2}{n+1}} + \| \rhos_k \|_{L^1}^{\frac{2n}{n+4}} + \| \rhos_k \|_{L^1}^{\frac{4}{n}} + \delta \Big) \Bigg) \| \nabla \rhos_k \|_{L^2}^2 \\
    &\qquad  + C(n,k)(\delta^{-1} \|\nabla W\|_{L^n}+1).
\end{align*}
Due to \eqref{eq:equiintegrability from entropy}, $\|\rhos_k\|_{L^1}$ can be made arbitrarily small for large $k$. Thus we can find a sufficiently large $k = k(n, E[\rho_0], \mathcal N_2[\rho_0], \| W \|_{L^\infty})$ and a sufficiently small $\delta=\delta(n,\|\nabla W\|_{L^n})$, such that for such $\delta$ and $k$,
\begin{align*}
    \frac{\diff}{\diff \tau} \int_{\mathbb R^n} \rhos_k^2 &\le -  \| \nabla \rhos_k \|_{L^2}^2  + C(n,k, \|\nabla W\|_{L^n})\\
    &\le - c(n) %
    \| \rhos_k \|_{L^2}^{\frac{2(n+2)}n}   + C(n,k, \|\nabla W\|_{L^n}),
\end{align*}
where the second inequality follows from \eqref{eq:GN H1 and L1 to L2} and the fact that $\|\rhos_k\|_{L^1}\leq 1$. Therefore, $X(\tau):= \|\rhos_k(\tau)\|_{L^2}^2$ satisfies the differential inequality  $$\dot X \le - c_1 X^{\frac{n+2}{n}} + C_2$$ with $c_1=c(n)$ and $C_2 = C(n, k, \|\nabla W\|_{L^n})$, thus $X(\tau)$ is decreasing whenever $X \ge (C_2/c_1)^{\frac{n}{n+2}}$. In other words, $X(\tau)$ has the upper bound $X(\tau) \le \max \{ X(0), (C_2/c_1)^{\frac{n}{n+2}} \}$. This means that
\begin{equation*}
    \int_{\mathbb R^n} \rhos_k (\tau)^2 \le C(n, \| W \|_{L^\infty}, \|\nabla W\|_{L^n}, \| \rho_0 \|_{L^2}, \mathcal N_2 [\rho_0]),
\end{equation*}
where we used that $E[\rho_0]$ can be bounded above using $\|\rho_0\|_{L^2}$ and $\|W\|_{L^\infty}$. Through \eqref{eq:relation rhos and rhosk p < infty}, we obtain a uniform-in-time bound of $\|\rhos(\tau)\|_{L^2}$, finishing the proof of \eqref{eq_L2_uniform}.

\noindent \textbf{Step 3. Uniform-in-time $H^1$ bound.}
In the rest of the proof we aim to check \eqref{eq_H1_uniform}, where it suffices to control the time-evolution of $\|\nabla \rhos(\tau)\|_{L^2}^2$. 
Taking its time derivative gives
\begin{equation}\label{eq_H1_time}
    \frac 1 2 \frac{\diff }{\diff \tau} \int_\Rd |\nabla \rhos|^2 = - \int_\Rd \Delta \rhos  (\Delta \rhos + \nabla \rhos \cdot \nabla (\Ws * \rhos) + \rhos \Delta (\Ws* \rhos) + \diver ( \rhos y ) ) =: - \|\Delta \rhos \|_{L^2}^2 - \sum_{i=1}^3 J_i.
\end{equation}
For $J_1 := \int_{\mathbb{R}^n} \Delta\rhos \nabla \rhos \cdot\nabla (\Ws* \rhos)$, using the fact that $\|\nabla W\|_{L^n} = \|\nabla \Ws\|_{L^n}$, we have
\[
\begin{split}
|J_1| %
    &\le \|\Delta \rhos \|_{L^2} \| \nabla \rhos \|_{L^2} \|\nabla W\|_{L^n} \| \rhos \|_{L^{\frac{n}{n-1}}}.
\end{split}
\]
Applying the Gagliardo-Nirenberg inequalities 
(see, e.g., \cite{Leoni2009Sobolev})
\[
\|\nabla\rhos\|_{L^2}\leq \|\Delta\rhos\|_{L^2}^{\frac{1}{2}}\|\rhos\|_{L^2}^{\frac{1}{2}} \quad\text{ and }\quad
    \|\rhos\|_{L^\frac{n}{n-1}}\leq C(n) \|\Delta\rhos\|_{L^2}^{\frac 2{n+4}} \|\rhos\|_{L^1}^{\frac {n+2}{n+4}},\label{eq_GN22}
\]
where we recall the well-known fact that $\| D^2 \rhos \|_{L^2} \le C(n) \| \Delta \rhos \|_{L^2}$,
the inequality for $J_1$ becomes
\[
|J_1| \leq   C(n)\|\nabla W\|_{L^n} \| \rhos \|_{L^2}^{\frac 1 2} \|\Delta \rhos \|_{L^2}^{1 + \frac 1 2 + \frac 2 {n+4}}  ,%
\]
where we also use that $\|\rhos\|_{L^1}=1$. Note that the power of $\| \Delta \rhos \|_{L^2}$ on the right hand side is strictly less than 2. Likewise,  $J_2 := \int_{\mathbb{R}^n} (\Delta\rhos) \rhos \Delta (\Ws* \rhos)$ satisfies
\begin{align*}
     |J_2| %
     &\le \|\Delta \rhos \|_{L^2} \|\rhos \|_{L^2} \| \Delta (\Ws* \rhos)\|_{L^\infty}.
\end{align*}
The last term on the right hand side can be controlled as
\begin{align*}
    \| \Delta (\Ws* \rhos)\|_{L^\infty} &\le \sum_{i=1}^n \left \| \frac{\partial \Ws}{\partial x_i} * \frac{\partial \rhos}{\partial x_i} \right \|_{L^\infty} \le \sum_{i=1}^n \left \| \frac{\partial \Ws}{\partial x_i} \right\|_{L^n} \left \| \frac{\partial \rhos}{\partial x_i} \right\|_{L^{ \frac{n}{n-1} }} \\ 
    &\le C(n) \| \nabla W \|_{L^n} \|\Delta \rhos \|_{L^2} ^{\frac {4}{n+4}} \|\rhos\|_{L^1}^{ \frac{n}{n+4} },
\end{align*}
where the second inequality follows from the Gagliardo-Nirenberg inequality $$\|\nabla\rhos\|_{L^\frac{n}{n-1}}\leq
C(n)
\|\Delta\rhos\|_{L^2}^{\frac{4}{n+4}}\|\rhos\|_{L^1}^{\frac{n}{n+4}}.$$
Thus
\begin{equation*}
    |J_2| \le C(n) \| \nabla W\|_{L^n}  \| \rhos \|_{L^2}  \|\Delta \rhos\|_{L^2}^{ 1 + \frac 4{n+4} },
\end{equation*}
and again the power of $\|\Delta \rhos\|_{L^2}$ is less than $2$. Finally, the term $J_3 := \int_\Rd \Delta \rhos \diver (\rhos y) \diff y$ can be explicitly computed as
\begin{align*}
    J_3 
    &=\sum_{i=1}^n \int_\Rd \frac{\partial^2 \rhos}{\partial x_i^2} \frac{\partial }{\partial x_i} ( \rhos x_ i) + \sum_{i\ne j} \int_\Rd \frac{\partial^2 \rhos}{\partial x_j^2} \frac{\partial }{\partial x_i} ( \rhos x_ i) \\
    &= \sum_{i=1}^n \frac 1 2 \int_\Rd \frac{\partial }{\partial x_i} \left( \frac{\partial \rhos }{\partial x_i}  \right)^2 x_ i  - \sum_{i=1}^n \int_\Rd \left( \frac{\partial \rhos}{\partial x_i} \right)^2  + \sum_{i\ne j} \int_\Rd \frac{\partial^2 \rhos}{\partial x_i \partial x_j} \frac{\partial \rhos }{\partial x_j} x_i \\
    &= \left(-1-\frac{n}{2}\right)\int_\Rd |\nabla\rhos|^2 \diff y,
\end{align*}
thus
\begin{equation*}
    |J_3|  \le C(n) \| \nabla \rhos \|_{L^2}^2 \le C(n) \|\Delta \rhos\|_{L^2} \|\rhos\|_{L^2}.
\end{equation*}
Since in the estimates for $J_1, \cdots, J_3$, the powers of $\|\Delta\rhos\|_{L^2}$ are all strictly lower than 2, plugging the estimates into \eqref{eq_H1_time} and applying Young's inequality for products gives
\begin{align*}
     \frac{\diff }{\diff t} \int_\Rd |\nabla \rhos|^2 &\le - \frac{1}{2}\| \Delta \rhos \|_{L^2}^2 + C( n, \| \nabla W \|_{L^n}, \| \rhos (\tau) \|_{L^2}) \\
     &\le -C(n) \| \nabla \rhos \|_{L^2}^{\frac{2(n+4)}{n+2}} \|\rhos_0\|_{L^1}^{-\frac{4}{n+2}} + C( n, \| \nabla W \|_{L^n}, \| \rhos (\tau) \|_{L^2}).
\end{align*}
Since we already have the uniform-in-time bound of $\| \rhos (\tau) \|_{L^2}$ in Step 2, the above differential inequality yields the uniform-in-time $H^1$ bound \eqref{eq_H1_uniform}.
\end{proof}

\begin{remark}
\label{rem:propagation of Hk regularity}
    We expect that the propagation of $H^k$ regularity for any integer $k> 1$ follows from a similar procedure as Step 3, although the computation becomes more involved. We leave the computation to interested readers.%
\end{remark}

\subsection{Uniform-in-time bounds of \texorpdfstring{$C^\alpha$}{Calpha} norm}

In this subsection, we aim to derive the propagation of regularity via an alternative approach. Instead of tracking the evolution of some integral-based quantities such as the $L^2$ or $H^1$ norm, which has been done in a vast amount of literature, we will track the evolution of point-wise quantities such as the modulus of continuity. In the context of nonlocal PDEs, such idea has been  successfully used by Kiselev--Nazarov--Volberg \cite{Kiselev2007} to establish the global-wellposedness for the SQG equation with critical dissipation. 

Our approach is similar to \cite{Kiselev2007}: in order to show that $\rhos$ has a certain modulus of continuity for all times, we will carefully look at the first ``breakthrough'' time $\tau_0$ where the modulus of continuity is about to be violated, and aim to derive a contradiction. While \cite{Kiselev2007} constructed a piecewise modulus of continuity to treat the criticality of SQG equation, for our application to \eqref{eq:Fokker-Plack rho scaled} it turns out the simple H\"older continuity would work.

Throughout this paper, for any $f:\mathbb{R}^n\to\mathbb{R}$, we denote its H\"older seminorm $[f]_{C^\alpha}$ and H\"older norm $\|f\|_{C^\alpha}$ as follows:
\begin{equation*}
    [f]_{C^\alpha} := \sup_{x\ne y} \frac{|f(x) - f(y)|}{|x-y|^\alpha}, \qquad \| f \|_{C^\alpha} := \| f \|_{L^\infty} + [ f ]_{C^\alpha} .
\end{equation*}

\begin{theorem}
	\label{thm:moc}
	Let 
	$ W \in \mathcal W^{1,\infty} (\Rd)$,
	$\alpha \in (0,1)$, $\rho \in C_+ ([0, T); C^\alpha (\mathbb R^n))$ be a classical solution of \eqref{eq:ADE rho} and assume
	\begin{enumerate}[label=(\alph*)]
		\item $n\geq 2$, and $W$ satisfies $\|W\|_{L^\infty} \leq C_W$, 
		    $\|\nabla W\|_{L^n} \leq C_W$
		and 
			 $\| \Delta W \| _{  L^{\frac n 2} } \le C_W $ .
		\item $\rho_0 \in L^1_+(\mathbb{R}^n)$ satisfies that $\int \rho_0 = 1$, $\mathcal{N}[\rho_0]<\infty$, and $\|\rho_0\|_{C^\alpha}<\infty.$
	\end{enumerate}
	Then, the rescaled density $\rhos(\tau,y)$ defined in \eqref{def_tau}--\eqref{def_tilderho}  is $C^\alpha$ H\"older continuous uniformly in time, in the sense that 
	\begin{equation}\label{holder_final}
	\|\rhos(\tau)\|_{C^\alpha}\leq K(C_W, \alpha, n, \mathcal N_2 [\rho_0], \|\rho_0\|_{C^\alpha}) \quad\text{ for all }\tau\geq 0.
	\end{equation}
\end{theorem}

Before presenting the proof, let us first state and prove a simple lemma that will be useful in the proof. It shows that if a function has a bounded $C^\alpha$ seminorm, as well as an $L \log L$ bound, it must have an $L^\infty$ bound. 
\begin{lemma}
	\label{cor:moc + entropy implies boundedness}
	For any function $f:\mathbb{R}^n\to\mathbb{R}$ and $\alpha \in (0,1)$, if $[f]_{C^\alpha} \leq K$ for some $K> 1$ and $\int_\Rd f|\log f| \leq C_0$, then we have 
	\begin{equation}\label{L_infty_f}
	\|f\|_{L^\infty} \leq C(n,\alpha,C_0) K^{\frac{n}{n+\alpha}}(\log K ) ^{- \frac { \alpha}{n + \alpha }}.
	\end{equation}

\end{lemma}

\begin{proof}
    Let $A := \|f\|_{L^\infty}$, and it suffices to obtain an upper bound of $A$ when $A>2$. Take any $x_0\in\mathbb{R}^n$ such that $f(x_0)\geq \frac{3}{4}A$. Using the modulus of continuity $[f]_{C^\alpha}\leq K$, we have that 
$
f(x)\geq \frac{3A}{4} - K |x-x_0|^\alpha$ for all $x\in \mathbb{R}^n$, thus
\[
f(x)\geq \frac{A}{2}\quad\text{ for all }|x-x_0|\le r_0 = \left( \tfrac {A} {4 K} \right)^{ \frac 1 \alpha } . 
\]
Combining this with the bound $\int f |\log f| \leq C_0$ and the fact that $\frac{A}{2}>1$, we have
	\begin{equation*}
		C_0 \ge \int_{\Rd}  f|\log f| \diff x  
		\ge \int_{B(x_0,r_0) } f |\log f|\diff x
		\ge \omega_n \left( \tfrac {A} {4K} \right)^{ \frac n \alpha }  \tfrac A 2 \log \tfrac A 2,
	\end{equation*}
	where $\omega_n$ is the volume of a unit ball in $\Rd$. This inequality can be rewritten as
	\begin{equation}\label{ineq_temp}
		\left( \tfrac A 2 \right)^{ \frac {n+\alpha} \alpha } \log \left(  \left( \tfrac A 2 \right)^{\frac {n+\alpha} \alpha } \right) \le C (n , \alpha) C_0 K ^{\frac n \alpha}.
	\end{equation}
	Setting $a:= (A/2)^{\frac{(n+\alpha)}{\alpha}}$ and $b:=  C_1 K ^{\frac n \alpha}$, where $C_1 :=\max\{ C (n , \alpha) C_0 ,1\}$, the above inequality implies that $a\log a \leq b$. To bound $a$, it suffices to estimate the solution $\bar a$ to the equation $\bar a \log \bar a = b$ for $b>0$. (Note that the function $a\log a$ is increasing for $a>1$, thus $1\leq a<\bar a$.) Since $\log \bar a \le \bar a$ for $\bar a>1$, we have that $b = \bar a \log \bar a \le \bar a^2 $, hence $\log \bar a \ge \frac 1 2 \log b $. This leads to
	\begin{equation*}
	    a < \bar a = \frac{\bar a \log \bar a}{\log \bar a} = \frac{b}{\log \bar a} \le 2 \frac{b}{\log b}.
	\end{equation*}
    Plugging the definition of $a$ and $b$ into above, we have
    \begin{equation*}
	    \left( \frac A 2 \right)^{ \frac {n+\alpha} \alpha } \le 2 \frac{C_1 K^{\frac n \alpha}}{ \log ( C_1 K^{\frac n \alpha })} \le \frac{2\alpha}{n} C_1 \frac{K^{ \frac n \alpha }}{ \log K },
	\end{equation*}
	where in the second inequality we use the fact that $C_1=\max\{ C (n , \alpha) C_0 ,1\}\geq 1$. Solving this inequality yields \eqref{L_infty_f} and finishes the proof.
\end{proof}

\begin{remark}
Note that if we replace the $L\log L$ bound $\int f|\log f| \leq C_0$ in \Cref{cor:moc + entropy implies boundedness} by an $L^1$ bound $\|f\|_{L^1}\leq C_0$ instead, then an estimate very similar to \eqref{L_infty_f} would still hold, except that we would lose the  $(\log K ) ^{- \frac { \alpha}{n + \alpha }}$ factor. As we will see soon, this negative power of $\log K$ plays an essential role in the proof of \Cref{thm:moc}.
\end{remark}

\begin{proof}[Proof of \Cref{thm:moc}]
	By continuous dependence in $L^1$ of the initial data, without loss of generality we can assume that $\rho_0 \in \mathcal W^{2,\infty} (\Rd)$, and hence $\rho \in C([0,T]; \mathcal W^{2,\infty} (\Rd))$ with $\partial \rho / \partial t \in C([0,T] \times \mathbb R^n)$ for any $T>0$. 
	Note that these regularity properties are also inherited by $\rhos$, since it is a (smooth) rescaling of $\rho$ given by \eqref{def_tau}--\eqref{def_tilderho}. Once \eqref{holder_final} is proved for $\mathcal{W}^{2,\infty}$ initial data (note that the bound $K$ is independent of $\|\rho_0\|_{\mathcal{W}^{2,\infty}}$), the $L^1$ continuous dependence on initial data in \Cref{thm:existence} allow us to approximate a $C^\alpha$ initial data and pass to the limit.
	
	Recall that $\rhos$ solves the rescaled equation \eqref{eq:Fokker-Plack rho scaled}, and \Cref{cor_entropy} give a uniform-in-time $L\log L$ bound of $\rhos$, namely 
	\begin{equation}\label{bd_entropy}
	\int_{\mathbb{R}^n} \rhos(\tau)|\log\rhos(\tau)| \leq C_0(n,\|W\|_{L^\infty}, E[\rho_0], \mathcal{N}_2[\rho_0]).
	\end{equation}
	Using \eqref{eq:Fokker-Plack rho scaled} and the bound \eqref{bd_entropy}, our goal is to show that $[\rhos(\tau)]_{C^\alpha}\leq K$ for all $\tau\geq 0$, where $K>\|\rho_0\|_{C^\alpha}$ is a sufficiently large constant to be determined later, which depends on $n, \alpha, C_W, C_0$ and $\|\rho_0\|_{C^\alpha}$. Once this is shown, combining it with \eqref{bd_entropy} and applying \Cref{cor:moc + entropy implies boundedness} yields the $L^\infty$ bound of $\rho$, finishing the proof.
	
	Towards a contradiction, assume that $[\rhos(\tau)]_{C^\alpha}\leq K$ is not satisfied for all $\tau\geq 0$. Let us set 
	\begin{equation*}
		\omega(r) := K r^\alpha\quad\text{ for }r\geq 0,
	\end{equation*} and define $\tau_0$ as the first time such that the modulus of continuity $\omega$ is about to be violated, i.e.,
	\begin{equation*}
		\tau_0 := \inf \{   \tau \geq 0 : \text{ there exist } y_1 \ne y_2 \text{ such that }  |\rhos(\tau,y_1) - \rhos(\tau ,y_2) | > \omega ( | y_1 - y_2 |)      \}.
	\end{equation*}Note that $\tau_0>0$ since $\|\rho_0\|_{C^\alpha}<K$ and $\tilde\rho \in C([0,T];C^\alpha)$ for any $T>0$. 
	Assuming $\tau_0<\infty$, let us take a closer look at $\rhos$ at the ``breakthrough'' time $\tau_0$ and derive various estimates on $\rhos(\tau_0)$ in the next 4 steps, and we will finally obtain a contradiction at the end of Step 4.

	\noindent{\textbf{Step 1.}} We claim that 
	\begin{equation}\label{moc_temp}
				|\rhos(\tau_0,y_1) - \rhos(\tau_0 ,y_2) | \le \omega ( | y_1 - y_2 |) \quad\text{ for all } y_1, y_2 \in \mathbb R^n,
	\end{equation}
	and there exist $z_1, z_2 \in \Rd$ such that $z_1\neq z_2$, and
	\begin{gather}
		\label{eq:moc proof 1}
		\rhos (\tau_0, z_1) - \rhos (\tau_0, z_2) = \omega(|z_1 - z_2|)  \\
		\label{eq:moc proof 2}
		\frac{ \partial \rhos}{\partial \tau} (\tau_0, z_1) - \frac{ \partial \rhos }{\partial \tau} (\tau_0, z_2) \ge 0.
	\end{gather}
    
    Indeed, by definition of $\tau_0$, \eqref{moc_temp} holds when $\tau_0$ is replaced by any $\tau<\tau_0$, thus \eqref{moc_temp} also holds at $\tau_0$ due to the continuity of $\rhos$ in time.
Also by definition of $\tau_0$, there exists a sequence $(\tau_k, y_1^{(k)} , y_2^{(k)})$ of points such that $\tau_k\in(\tau_0,\tau_0+1)$, $\tau_k\searrow\tau_0$, and
	\begin{equation}\label{def_yk}
	 \omega ( | y_1^{(k)} - y_2^{(k)} |) <   	\rhos(\tau_k,y_1^{(k)}) - \rhos(\tau_k ,y_2^{(k)})   \le  \| \nabla \rhos (\tau_k, \cdot ) \|_{L^\infty} | y_1^{(k)} - y_2^{(k)} | \leq \tilde{C_1}| y_1^{(k)} - y_2^{(k)} |,
	\end{equation}
	where $\tilde{C_1}:= \sup_k \| \nabla \rhos (\tau_k, \cdot ) \|_{L^\infty} < \infty$ since $\tilde\rho \in C([0,\tau_0+1];\mathcal{W}^{2,\infty})$.
	Using that $\omega(r)=Kr^\alpha$, the above inequality becomes
	\begin{equation}\label{min_dist}
		 | y_1^{(k)} - y_2^{(k)} | \ge (K\tilde C_1^{-1})^{\frac{1}{1-\alpha}} > 0,%
	\end{equation}
	so $y_1^{(k)} - y_2^{(k)} \not \to 0$. On the other hand, using $\rhos \ge 0$ we have that
	\begin{equation*}
		\omega ( | y_1^{(k)} - y_2^{(k)} |) <   	\rhos(\tau_k,y_1^{(k)}) - \rhos(\tau_k ,y_2^{(k)}) \le \|  \rhos  (\tau_k , \cdot) \|_{L^\infty } \leq \tilde C_2,
	\end{equation*}
	where $\tilde{C_2}:= \sup_k \| \rhos (\tau_k, \cdot ) \|_{L^\infty} < \infty$ again due to $\tilde\rho \in C([0,\tau_0+1];\mathcal{W}^{2,\infty})$.
	This leads to the estimate
	\begin{equation}\label{max_dist}
		 | y_1^{(k)} - y_2^{(k)} | \le \left(  \tilde C_2 K^{-1} \right) ^{\frac 1 {\alpha}} ,
	\end{equation}
	meaning that $y_1^{(k)}$ and $ y_2^{(k)}$ cannot be too far apart either.
	
	To obtain a convergent subsequence of $(y_1^{(k)}, y_2^{(k)})$, we need to show that the sequence is uniformly bounded in $k$. Towards this end, recall that the second moment $\mathcal{N}_2[\rhos]$ is known to be bounded by \Cref{thm_m2}, and we will use this to show that $\{y_1^{(k)}\}$ are uniformly bounded. First note that $\rhos(\tau_k,y_1^{(k)})$ is uniformly positive since
	 \begin{equation*}
	     \rhos (\tau_k, y_1^{(k)}) > \omega ( | y_1^{(k)} - y_2^{(k)} |) \ge K (K\tilde C_1^{-1})^{\frac{\alpha}{1-\alpha}} =: c_0>0,
	 \end{equation*}
	 where the second inequality follows from \eqref{min_dist}.
	 As a result, by definition of $\tilde C_1= \sup_k \| \nabla \rhos (\tau_k, \cdot ) \|_{L^\infty}$, we have 
	 using the mean value theorem
	 \[
	 \rhos (\tau_k, y)\geq \frac{c_0}{2}\quad\text{ for all }y\in B\Big(y_1^{(k)},\frac{c_0}{2\tilde C_1}\Big).
	 \]
	 Combining this with the uniform bound of the second moment in \Cref{thm_m2} provides an upper bound of $|y_1^{(k)}|$ independent of $k$.
	 Notice that
	  	 \begin{equation*}
 	     \int_\Rd |y|^2 \rho(\tau_k,y) \diff y \ge \frac {c_0}2 \min_{B(y_1^{(k)},\frac{c_0}{2\tilde C_1})}  |y|^2.
 	 \end{equation*}
	 Thus $|y_2^{(k)}|$ are also uniformly bounded due to \eqref{max_dist}.
	Hence, there exists a convergent subsequence of $(y_1^{(k)}, y_2^{(k)})$, and let its limit be $z_1$ and $z_2$. Note that $z_1\neq z_2$ due to \eqref{min_dist}. Using the first inequality in \eqref{def_yk} and passing to the limit, we have that $\rhos(\tau_0,z_1)-\rhos(\tau_0,z_2)\geq \omega(|z_1-z_2|)$, and combining it with \eqref{moc_temp} yields \eqref{eq:moc proof 1}.
	
	Finally, to show \eqref{eq:moc proof 2}, recall that for any $h\in (0,\tau_0)$ we know that $\rhos(\tau_0-h, \cdot)$ has modulus of continuity $\omega$. Combining this with \eqref{eq:moc proof 1} gives that
	\begin{equation*}
		\frac{\rhos( \tau_0 , z_1 ) - \rhos( \tau_0 - h , z_1 )  }{h} - \frac{\rhos( \tau_0, z_2 ) - \rhos( \tau_0 - h , z_2 )  }{h} \ge \frac{\omega(|z_1-z_2|) - \omega(|z_1-z_2| ) }{h} = 0.
	\end{equation*}
 	Passing to the limit as $h \to 0^+$ finishes the proof of \eqref{eq:moc proof 2}.

	\medskip 

	\noindent{\textbf{Step 2.}} Set $r_0 := |z_1 - z_2|$ and assume WLOG that 
	$
		z_1 - z_2 = r_0 \vv e_1.
	$
	In this step we aim to prove the following:
	\begin{gather}
		\label{eq:moc proof 6}
		\nabla \rhos (\tau_0, z_1) = \omega' (r_0) \vv e_1 = \nabla \rhos (\tau_0, z_2)  , \\
		\label{eq:moc proof 7}
		\partial_{11} \rhos (\tau_0, z_1) \le  \omega'' (r_0), \qquad \partial_{11}  \rhos (\tau_0, z_2)  \ge - \omega'' (r_0), \\
		\label{eq:moc proof 7b}
		\partial_{ii} \rhos (\tau_0, z_1) - \partial_{ii} \rhos (\tau_0, z_2) \le 0 \quad\text{ for  } i=2,\dots,n.
	\end{gather}
	To show \eqref{eq:moc proof 6} and \eqref{eq:moc proof 7}, define
	\begin{equation*}
		g(y) := \rhos (\tau_0, z_2) + \omega (| y - z_2 |) .
	\end{equation*}
	Since in step 1 we showed that $\rhos(\tau_0,\cdot)$ has modulus of continuity $\omega$ achieved at $z_1$ and $z_2$, it implies that $g(y)\geq \rhos(\tau_0,y)$ for all $y\in\Rd$, with equality achieved at $y=z_1$. This yields $\nabla\rhos(\tau_0, z_1) = \nabla g(z_1) $ and $\partial_{11}\rhos(\tau_0, z_1) \leq \partial_{11} g(z_1)$.
	A parallel argument can be applied similarly to $(\tau_0, z_2)$, which finishes the proof of \eqref{eq:moc proof 6} and \eqref{eq:moc proof 7}. Finally, to show \eqref{eq:moc proof 7b}, define
	\begin{equation*}
	    h(v) := \rhos(\tau_0, z_1 + v) - \rhos(\tau_0,z_2 + v).
	\end{equation*}
	Again, the fact that $\rhos(\tau_0,\cdot)$ has modulus of continuity $\omega$ achieved at $z_1$ and $z_2$ gives that $h(v)\leq \omega(|z_1-z_2|)$ for all $v\in\Rd$, and it achieves its maximum at $v = 0$. Thus we recover the estimate \eqref{eq:moc proof 7b} for $i =2,\dots,n$.
	Notice that this is valid also for $i = 1$, but in \eqref{eq:moc proof 7} we have better quantitative information.

	\noindent{\textbf{Step 3.}} Let us estimate $A :=  \| \rho (\tau_0, \cdot) \|_{L^\infty}$ and $r_0 := |z_1-z_2|$ in terms of $K$, which will be helpful for us to obtain a contradiction later. Namely, we will prove that
	\begin{gather}
		\label{eq:moc proof 8}
		A \defeq \| \rho (\tau_0, \cdot) \|_{L^\infty} \le C_1K^{\frac {n}{n + \alpha }} (\log K ) ^{-\frac {\alpha}{n + \alpha }}. \\
		\label{eq:moc proof 9} 
		r_0 \le C_2 K^{ - \frac 1 {n+\alpha} } (\log K)^{ - \frac{1}{n+\alpha}} .
	\end{gather}
	where $C_1, C_2 > 0$ depend only on $C_0, n , \alpha$.

	Estimate \eqref{eq:moc proof 8} directly follows from \Cref{cor:moc + entropy implies boundedness}, where we also used \eqref{bd_entropy}. Since $\rhos(\tau_0,z_1) = \omega(r_0) + \rhos(\tau_0, z_2) \ge \omega(r_0) + 0 = K r_0^\alpha$ we have that
	$
		r_0 \le A^{\frac 1 \alpha} K^{-\frac 1 \alpha}.
	$
	Combining this with \eqref{eq:moc proof 8} yields \eqref{eq:moc proof 9}.
	
	\noindent{\textbf{Step 4.}} In this step, we will show that 
	\[
	\frac{ \partial \rhos}{\partial \tau} (\tau_0, z_1) - \frac{ \partial \rhos }{\partial \tau} (\tau_0, z_2) <0
	\]if $K$ is sufficiently large (depending on $C_0,n,\alpha, C_W$), which would lead to a direct contradiction with \eqref{eq:moc proof 2}. Since $\rhos$ satisfies the rescaled equation \eqref{eq:Fokker-Plack rho scaled}, $\frac{ \partial \rhos}{\partial \tau} (\tau_0, z_1) - \frac{ \partial \rhos }{\partial \tau} (\tau_0, z_2)$ can be written as $T_1+T_2+T_3+T_4$, where the four terms are defined below. Our claim is that they satisfy the inequalities
	\begin{align}
		\label{eq:moc proof T1}
		T_1 := \Delta \rhos (\tau_0, z_1 ) - \Delta \rhos (\tau_0, z_2 )  &\le - C_3 K^{ \frac{n + 2}{n+\alpha}} (\log K)^{ \frac{2- \alpha}{ n + \alpha} } ,\\
		\label{eq:moc proof T2}
		T_2 := \diver  (  y \rhos )  (\tau_0, z_1 ) - \diver  (  y \rhos ) (\tau_0, z_2 )  &\le  C_4 K^{ \frac{n }{n+\alpha}} (\log K)^{ -\frac{\alpha}{ n + \alpha} }, \\		
		\label{eq:moc proof T3}
		T_3 := (\nabla \rhos \cdot \nabla (\Ws * \rhos ) )  (\tau_0, z_1 ) - (\nabla \rhos \cdot \nabla (\Ws * \rhos )) (\tau_0, z_2 )  &\le  C_5 K^{ \frac{  n+2}{n+\alpha} } (\log K)^{ \frac{-n\alpha+n-\alpha }{ n (n + \alpha)} } ,\\
		\label{eq:moc proof T4}
		T_4 := \rhos \Delta (\Ws * \rhos )   (\tau_0, z_1 ) -  \rhos \Delta (\Ws * \rhos ) (\tau_0, z_2 )  &\le  C_6 K^{ \frac{n+2 }{n+\alpha}} (\log K)^{ - \frac{\alpha( n+2) }{ n (n + \alpha) } } ,
	\end{align} 
	where $C_3, C_4 > 0$  depend only on $C_0, n , \alpha$ and  $C_5, C_6 > 0$ depend only on $C_0, n , \alpha, C_W$.
	
	To recover \eqref{eq:moc proof T1}, note that \eqref{eq:moc proof 7}-\eqref{eq:moc proof 7b} yields that $T_1 \leq 2\omega''(r_0)=2\alpha(\alpha-1) Kr_0^{\alpha-2}$, which is negative since $\alpha\in(0,1)$. Combining this with \eqref{eq:moc proof 9} gives 
	\begin{align*}
	    T_1 &\le 2 \alpha (\alpha - 1) K r_0^{\alpha - 2} \le - C_3 K \left( K^{ - \frac 1 {n+\alpha} } (\log K)^{ - \frac{1}{n+\alpha}} \right) ^{\alpha - 2} \leq - C_3 K^{ \frac{n + 2}{n+\alpha}} (\log K)^{ \frac{2- \alpha}{ n + \alpha} } .
	\end{align*}

	For \eqref{eq:moc proof T2}, we apply \eqref{eq:moc proof 6} and \eqref{eq:moc proof 9} to get
	\begin{align*}
	    T_2 &= \nabla \rhos (\tau_0, z_1) \cdot (z_1 - z_2) + n (\rhos (\tau_0,z_1) - \rhos(\tau_0,z_2) ) \\
	    &= \omega'(r_0) r_0 + n \omega(r_0) = (\alpha + n) K r_0^\alpha \leq  C_4 K^{ \frac{n }{n+\alpha}} (\log K)^{ - \frac{\alpha }{ n + \alpha } } .
	\end{align*}
	To compute \eqref{eq:moc proof T3} we use \eqref{eq:moc proof 6} to deduce
	\[
	T_3 \leq 2\omega'(r_0) \|\nabla\Ws * \rhos\|_{L^\infty}.
	\]
	We can then apply Young's convolution inequality to obtain
	\begin{equation}
	\label{eq:moc proof T3 estimate}
	    \| \nabla \Ws * \rhos (\tau_0) \|_{L^\infty} \le \|\nabla \Ws \|_{L^n} \| \rhos (\tau_0)\|_{L^{ \frac{n}{n-1}}} \le 
	    C_W
	    \| \rhos (\tau_0)\|_{L^1} ^{\frac {n-1} {n}} \|\rhos (\tau_0) \|_{L^\infty}^{\frac 1 {n}} \le 
	    C_W
	    A^{\frac 1 {n}} .
	\end{equation}
This lead to the bound
	\begin{align*}
	    T_3 &\le 
	    C_W
	    \omega' (r_0) A^{\frac 1 {n}} = C(\alpha,C_W) Kr_0^{\alpha-1} A^{\frac 1 {n}} 	
	\end{align*}
	and we conclude \eqref{eq:moc proof T3} by using \eqref{eq:moc proof 8} and \eqref{eq:moc proof 9}.

	To compute \eqref{eq:moc proof T4} we proceed similarly
	\begin{equation} 
	\label{eq:moc proof T4 estimate}
	\begin{aligned}
	    T_4 &\le 2 \|\rhos (\tau_0) \|_{L^\infty} \| \Delta (\Ws * \rhos) (\tau_0) \|_{L^\infty} 
	    \le 2 A \| \Delta \Ws \|_{L^{\frac n 2}} \|\rhos (\tau_0) \|_{L^{\frac{n}{n-2}}} \\
	    &
	    \le 2 A ^{1 + \frac 2 n} \|\Delta W \|_{L^{\frac n 2 }} \|\rhos (\tau_0) \|_{L^1}^{ \frac{n-2} n  }
	    \le C A^{1 + \frac 2 n},
	\end{aligned}
	\end{equation}
	and plugging \eqref{eq:moc proof 8} into this inequality yields \eqref{eq:moc proof T4}.

Finally, comparing the powers in $T_1,\cdots,T_4$, 
note that $|T_1|$ (coming from  $\Delta \rhos$) has the fastest growth as $K\to\infty$, since it has a larger power of $\log K$ compared to the powers of $T_3, T_4$. 
	By choosing $K$ large enough we have that $T_1 + T_2 + T_3 + T_4 < 0$. This is a contradiction with \eqref{eq:moc proof 2}.
\end{proof}

\begin{remark}
Note that in step 4, the ``good contribution from diffusion'' $T_1$ \eqref{eq:moc proof T1}  and the ``bad contribution from aggregation'' $T_3$  and $T_4$ \eqref{eq:moc proof T3}--\eqref{eq:moc proof T4}  carry exactly the same power of $K$, although they have different powers of $\log K$. This subtle difference in the logarithm powers is the key for us to show that $T_1$ dominates $T_3$ and $T_4$ for $K\gg 1$. In this sense, the a priori $L \log L$ bound in \Cref{cor_entropy} is playing a crucial role since it contributes the logarithm term in \Cref{cor:moc + entropy implies boundedness}. Also, the assumptions on $\|\nabla W\|_{L^n}$ and $\|\Delta W\|_{L^{n/2}}$ are sharp in the sense that if the assumptions were to be made in $L^p$ spaces with any lower $p$, it would result in a higher power of $K$ in \eqref{eq:moc proof T3} and \eqref{eq:moc proof T4}, and the proof would not go through since $T_1$ would not dominate $T_3$ and $T_4$ for $K\gg 1$.
\end{remark}

\section{Convergence to the Gaussian}
\label{sec:intermediate asymptotics}
In this section we focus on obtaining the asymptotic behaviour based on the uniform estimates in the previous two sections.
We first concentrate on the $L^1$ relative entropy approach as introduced in the linear Fokker-Planck equation in \cite{To99,AMTU2001} based on the crucial use of the logarithmic Sobolev inequality.
As usual the $L^2$ relative entropy strategy can also be applied similarly, replacing the log-Sobolev by Poincaré's inequality with respect to the Gaussian measure (see \cite{AMTU2001}).
For an elementary presentation in this direction we send the reader to \cite{Vazquez2017}.

\subsection{\texorpdfstring{$L^1$}{L1} relative entropy}
\label{sec:L1 relative entropy}
Going back to the notion of $L^1$ relative entropy given by \eqref{eq:L1 entropy}, we can 
can compute the time derivative as
	\begin{align*}
		\frac{\diff}{\diff \tau}	E_1 (\rhos \| G)  &= \int_\Rd \left ( \log \rhos + 1 + \frac 1 2 |y|^2 \right ) \rhos _t 
		=- \int_\Rd  \nabla \left ( \log \rhos + \frac 1 2 |y|^2 \right ) \cdot  ( \nabla \rhos + \rhos y + \rhos \nabla \Ws * \rhos  )
		\\
		&
		=:- I_1 (\rhos \| G ) - J_1 - J_2 ,
	\end{align*}
	where $I_1$ is the relative Fisher information 
\begin{equation*}
	I_1 (\rhos \| G ) = \int_\Rd \rhos \left| \nabla \log \rhos + y \right|^2 \diff y =  \int_\Rd \rhos \left| \nabla \log \frac \rhos G \right|^2 \diff y ,
\end{equation*}
and $J_1$ and $J_2$ are given by
\begin{align*}
	    J_1 :=  \int_\Rd \rhos y \cdot \nabla \Ws * \rhos \qquad \text{ and } \qquad 
	    J_2 := \int_\Rd    \nabla  \rhos \cdot \nabla \Ws * \rhos = -\int_\Rd      \rhos \Delta \Ws * \rhos.
\end{align*}
\begin{remark}
For the heat equation i.e. $W = \Ws = 0$, the above becomes
$
	\frac{\diff }{\diff\tau} E_1 (\rhos \| G) = - I_1 [\rhos \| G] .
$
From the logarithmic Sobolev inequality (see \cite{Gross1975}), it is classical that, when $W = 0$ we have
\begin{equation} \label{LSI}
	E_1(\rhos \| G ) \le \frac 1 2  I_1 (\rhos \| G ) .
\end{equation}
From which $\dot E_1 \le - 2 E_1$ and we recover the exponential decay $E_1(\rhos \| G) \le E_1 (\rhos_0 \| G ) e^{-2\tau}$.
\end{remark}

\begin{remark}
\label{rem:CK}
Applying the Csiszar-Kullback inequality \cite{K59,Cs63,AMTU2001}, since $1 + 2t = e^{2\tau}$, we have
\begin{equation} \label{CK}
	E_1(\rhos \|G) \le C(1+\tau)^\beta e^{-\alpha \tau} \implies \| \rho - U \|_{L^1} = \| \rhos - G \|_{L^1} \le 2 \sqrt{ E_1 (\rhos \| G)} \le C  (1 + \ln(1+2t))^{\frac \beta 2 } t^{-\frac \alpha 4}.
\end{equation}
where 
$
    U(t,x) = K\left (t + \frac 1 2 , x \right ).
$
Using the standard decay of the heat equation for even initial data with bounded second moment, we know that $\|U(t) - K(t) \|_{L^1} \le Ct^{- 1}$ (see \cite{duoandikoetxea1992moments}). This means that, as long as $\alpha < 1$, we can always replace $U$ by $K$ as an intermediate asymptotics profile
preserving the rate.
Notice also that one can get the convergence in 2-Wasserstein distance by using Talagrand inequality (see, e.g., \cite{MR1760620,CMV03}). It is a challenging problem to decide whether the Fisher information $I_1(\rhos \|G)$ also decays to $0$ as $\tau \to \infty$ with an explicit rate.
\end{remark}

To prove our convergence results, we will show that
$|J_i| \le C_i e^{-\alpha_i \tau}$ for some $\alpha_i>0$ under certain assumption on $W$. Once this is shown, using the logarithmic Sobolev inequality, we recover
\begin{align}
    \label{eq:E1 general estimate 1}
		\frac{\diff}{\diff \tau}	E_1 (\rhos \| G) & \le - 2 E_1 (\rhos \| G ) + C_1  e^{-\alpha_1 \tau} + C_2 e^{-\alpha_2 \tau}   .
\end{align}
Solving the differential inequality, we conclude that
\begin{equation}
    \label{eq:E1 general estimate 2}
\begin{aligned}
		E_1 (\rhos \| G) \le e^{-2\tau}	E_1 (\rhos_0 \| G) + &\,C_1 F_{\alpha_1} (\tau) + C_2 F_{\alpha_2} (\tau), \\
		&\text{ where }
		F_\alpha (\tau) = e^{-2\tau} \int_0^\tau e^{(2-\alpha)s} \diff s \le \begin{dcases}
		    \frac 1{\alpha -2} e^{-2\tau} & \alpha > 2, \\
		    \tau e^{-2\tau} & \alpha = 2, \\
		    \frac 1{2-\alpha} e^{-\alpha \tau} & \alpha < 2.
		\end{dcases}
	\end{aligned}
\end{equation}

With this approach, it remains to obtain the best possible rate of decay in $J_1$ and $J_2$. In \eqref{eq:scaling classical Sobolev}, one can easily check that $\| D^\alpha \Ws(\tau) \|_{L^q} = e^{(|\alpha| - \frac{n}q) \tau}\| D^\alpha W \|_{L^q}$ has the fastest decay when $\alpha$ is the smallest (i.e. 0) and $q$ is the lowest (i.e. 1). Therefore, the best possible decay of $J_1$ and $J_2$ is obtained by moving all derivatives away from $\Ws$, and only let $\|\Ws(\tau)\|_{L^1}$ appear in the estimate (note that it requires $\|W\|_{L^1}$ be finite). Since $\| \Ws(\tau) \|_{L^1} = e^{-n\tau} \|W \|_{L^1}$, $J_1$ and $J_2$ also decay with this rate (the detailed proof will be done in \Cref{n1_conv}). Plugging this into the inequality \eqref{eq:E1 general estimate 2} for $E_1$, we get the decays: $E_1\leq e^{-\tau}$ if $n=1$, $\tau e^{-2\tau}$ if $n = 2$, and $e^{-2\tau}$ if $n \ge 3$. It is a challenging open problem to prove or disprove if these decay rates are sharp in dimensions $n=1,2$ under the assumptions of \Cref{n1_conv}.

In the next two theorems, we will use two different ways to prove the decay of $E_1$ under different assumptions of $W$. We first prove \Cref{n1_conv} assuming $W\in L^1$, which leads to the best possible rate of decay using the argument in the previous paragraph. However, the assumption $W\in L^1$ is a bit too restrictive, since it requires $W$ to have fast decay at infinity. We then prove \Cref{thm:convergence to the Gaussian L1} with weaker assumptions on $W$, where $W$ is allowed to have arbitrarily slow power-law decay such as $W(|x|)\sim |x|^{-\ee}$ for $|x|\gg 1$ for any $\ee>0$. This is done at the expense of a slower convergence rate; in \Cref{rmk_conv} we will explain why it is natural to expect slower convergence when $W$ has slower decay at infinity.

\begin{theorem}\label{n1_conv}
    Let $n\geq 1$. Assume  $W\in \mathcal W^{1,\infty}(\Rd)\cap L^1(\Rd)$ with $\nabla W \in L^n(\Rd)$. If $n\geq 2$, further assume that $\Delta W \in L^\frac{n}{2}(\Rd)$. 
    Suppose
    $ \rho_0\in L^1_+(\mathbb{R}^n)$ with $\int_{\mathbb{R}^n}\rho_0 \diff x = 1$, $E[\rho_0]<\infty$, and $\mathcal{N}_2[\rho_0]<\infty$.  Let $\rho(x,t)$ be the solution constructed in \Cref{thm:existence} with initial data $\rho_0$. Then the rescaled density $\rhos(\tau,y)$ defined in \eqref{def_tau}--\eqref{def_tilderho} satisfies
    \[
        E_1 (\rhos \| G) \le \begin{cases}
        C e^{-\tau} & n=1,\\
        C (1+\tau)  e^{-2\tau} & n=2,\\
        C  e^{-2\tau} & n\geq3,
        \end{cases}
    \]
    where $C<\infty$ depends on $\rho_0$ and $W$.
    In addition, $|\mathcal N_2 [\rhos(\tau)] - n|$ also has exponential decay in $\tau$, with the same upper bound as in $E_1$.

\end{theorem}
\begin{proof}
From the instantaneous regularisation result in \Cref{thm:instant regularisation}, the solution $\rhos$ of \eqref{eq:Fokker-Plack rho scaled} is in  $C((0, \infty),\mathcal W^{k,p}(\Rd))$ for any $k\ge 0$ and $p\in [1, \infty]$. In particular, we have $\rhos(1,\cdot)\in H^1(\Rd)\cap C^\alpha(\Rd)$ for any $\alpha\in(0,1)$. Applying the uniform-in-time propagation of $H^1$ regularity proved in \Cref{thm:propagation of L2 and H1 regularity} with $\tau=1$ being the initial time yields that 
\begin{equation}
    \label{eq_h1_temp}
\sup_{\tau\geq 1}\|\rhos(\tau)\|_{H^1}<C(n,\|W\|_{L^\infty}, \|\nabla W\|_{L^n}, \|\rhos(1)\|_{H^1}, \mathcal{N}_2[\rhos(1)])<\infty\quad\text{ for }n\geq 1.
\end{equation}
In other words, we can also say that $C$ depends on $\rho_0$ and $W$ in a quite non-explicit manner.

For $n=1$, combining the Gagliardo-Nirenberg inequality
$
\| \rhos \|_{L^\infty} \le \|\nabla \rhos \|_{L^2}^{\frac 2 3} \|\rhos \|_{L^1}^{\frac 1 3}$ with \eqref{eq_h1_temp} directly yields that $\sup_{\tau\geq 1}\|\rhos(\tau)\|_{L^\infty}< C$. 

For $n\geq 2$, note that $H^1(\Rd)$ is not embedded in $L^\infty(\Rd)$. To show that $\sup_{\tau\geq 1}\|\rhos(\tau)\|_{L^\infty}< C$ for $n\geq 2$, one way is to obtain uniform-in-time propagation of the $H^k$ regularity, which we will not prove here (see \Cref{rem:propagation of Hk regularity}). Instead, let us apply the uniform-in-time  propagation of $C^\alpha$ regularity in \Cref{thm:moc} with $\tau=1$ being the initial time. It yields that for any $\alpha\in(0,1),$
\begin{equation*}
\sup_{\tau\geq 1}\|\rhos(\tau)\|_{C^\alpha}<C(n,\alpha,\|W\|_{L^\infty}, \|\nabla W\|_{L^n},\|\Delta W\|_{L^\frac{n}{2}} \|\rhos(1)\|_{C^\alpha}, \mathcal{N}_2[\rhos(1)])<\infty\quad\text{ for }n\geq 2,
\end{equation*}
which directly yields that $\sup_{\tau\geq 1}\|\rhos(\tau)\|_{L^\infty}< C$ for $n\geq 2$. We point out that so far we have not used $W\in L^1(\Rd)$.

Now that we have obtained the uniform-in-time  bounds (for all $\tau>1$) of the $L^\infty$ and $H^1$ norms of $\rhos$ for all $n\geq 1$, we will use these to prove the decay with $J_1$ and $J_2$ with the optimal rate when $W\in L^1(\Rd)$. For $J_1$, using Hölder's and Young's inequalities we have
\begin{equation}\label{J1_1d}
\begin{split}
    |J_1| &\le  \left( \int |\rhos|^2 |y|^2 \right)^{\frac 1 2} \| \nabla \Ws * \rhos \|_{L^2}  \le \| \rhos \|_{L^\infty}^{\frac 1 2} \mathcal N_2 [\rhos] ^{\frac 1 2} \| \nabla \Ws * \rhos \|_{L^2}  \\
    & \le \| \rhos \|_{L^\infty}^{\frac 1 2} \mathcal N_2 [\rhos] ^{\frac 1 2} \|\Ws\|_{L^1} \| \nabla \rhos \|_{L^2} 
    \le C e^{-n\tau},
\end{split}
\end{equation}
where in the last inequality we use the uniform-in-time bound on $\mathcal{N}_2[\rhos(\tau)]$ in \Cref{thm_m2}, as well as the fact that $\|\Ws\|_{L^1}=\|W\|_{L^1}e^{-n\tau}$ from \eqref{eq:scaling classical Sobolev}.
Likewise we can estimate $J_2$ as
\begin{equation}
    |J_2|\le \|\nabla \rhos\|_{L^2} \| \nabla \Ws * \rhos \|_{L^2} \le \| \nabla \rhos \|_{L^2}^2 \| \Ws \|_{L^1} \le C e^{-n\tau}.
\end{equation}
Plugging these estimates on $J_1$ and $J_2$ into \eqref{eq:E1 general estimate 1}, we obtain \eqref{eq:E1 general estimate 2} with $\alpha_1=\alpha_2=n$, finishing the proof for $E_1 (\rhos \| G)$.

Finally, note that the convergence of $|\mathcal N_2 (\tau)-n|$ immediately follows from the above estimates for $J_1$. In fact, from \eqref{dN2dt} we have 
\[
\frac \diff{\diff\tau} (\mathcal N_2 - n) = -2(\mathcal N_2 - n) - 2 J_1,
\]
and solving this differential equation gives
\[\mathcal N_2 (\tau) - n = e^{-2(\tau -1)} (\mathcal N_2 (1) - n) - 2 e^{-2\tau} \int_1^\tau e^{2\tau'} J_1(\tau') \diff\tau'.
\]
Using
\eqref{J1_1d} into the right hand side gives the exponential decaying bound of $|\mathcal N_2 [\rhos]-n|$, finishing the proof.
\end{proof}

\begin{theorem}
	\label{thm:convergence to the Gaussian L1}
Let $n\geq 1$. Assume $ W \in \mathcal W^{1,\infty} (\Rd)$ satisfies $\nabla W\in L^n(\Rd)$. If $n\geq 2$, further assume that $\Delta W \in L^\frac{n}{2}(\Rd)$. 
Suppose $ \rho_0\in L^1_+(\mathbb{R}^n)$ with $\int_{\mathbb{R}^n}\rho_0 \diff x = 1$, $E[\rho_0]<\infty$, and $\mathcal{N}_2[\rho_0]<\infty$.  Let $\rho(x,t)$ be the solution constructed in \Cref{thm:existence} with initial data $\rho_0$. Then the rescaled density $\rhos(\tau,y)$ satisfies the following:
\begin{enumerate}
\item[(a)] For $n=1$, if in addition we assume that $(-\Delta )^{\frac{1}{2}-\ee}W \in L^1(\mathbb{R})$ for some $\ee \in (0,\frac{1}{2})$, then for all $\tau\geq 1$ we have that
\begin{equation}
	\label{eq_e1_1d}
    E_1 (\rhos \| G) \le C \Big ( e^{-2\tau } + \| (-\Delta )^{\frac{1}{2}-\ee} W \|_{L^{1}} F_{ 2\ee} (\tau)  \Big ) \leq Ce^{-2\ee \tau}
\end{equation}
\[
|\mathcal N_2 [\rhos] - n| \leq Ce^{-2\ee \tau},
\]
where $C<\infty$ depends on $\rho_0$ and $W$.

	    \item[(b)] For $n= 2$, if in addition we assume that $\nabla W \in L^{p_1}(\mathbb{R}^2)$ for some $p_1\in[1,2)$, then for all $\tau \ge 1$ we have that
	\begin{equation}
	\label{eq_e1_temp1}
    E_1 (\rhos \| G) \le C \Big ( e^{-2\tau } + \| \nabla W \|_{L^{p_1}} F_{ \frac{2}{p_1} - 1} (\tau)  \Big ) \leq Ce^{(\frac{2}{p_1}-1)\tau}
\end{equation}
\[
|\mathcal N_2 [\rhos] - n|  \leq Ce^{(\frac{2}{p_1}-1)\tau},
\]
where $C<\infty$ depends on $\rho_0$ and $W$.

\item[(c)] For $n\geq 3$, if in addition we assume that $\nabla W \in L^{p_1}(\Rd)$ and $\Delta W \in L^{p_2}(\Rd)$ for some $p_1\in[1,n)$ and $p_2 \in [1,\frac{n}{2})$, then for all $\tau \ge 1$ we have that
	\begin{equation}
	\label{eq:convergence to the Gaussian L1 estimate 1}
    E_1 (\rhos \| G) \le C \Big ( e^{-2\tau } + \| \nabla W \|_{L^{p_1}} F_{ \frac{n}{p_1} - 1} (\tau) + \|\Delta W\|_{L^{p_2}} F_{\frac{n}{p_2}-2}(\tau) \Big ), 
\end{equation}
\[
|\mathcal N_2 [\rhos] - n| \le C \big ( e^{-2\tau } + \| \nabla W \|_{L^{p_1}} F_{ \frac{n}{p_1} - 1} (\tau) \big),
\]
where $C<\infty$ depends on $\rho_0$ and $W$.
\end{enumerate}
In particular, for $n\geq 2$
if $ W \in \mathcal W^{1,\infty} (\Rd)$, $\nabla W \in  L^{n-\ee} (\Rd)$, $\Delta W \in  L^{\frac n 2} (\Rd)$ (and also $\Delta W \in  L^{\frac n 2 - \ee} (\Rd)$ if $n\geq 3$) for some $\ee > 0$,
and $\rho_0$ satisfies the above assumptions, then the rescaled solution $\rhos$ satisfies
$E_1 (\rhos \| G) \to 0$ and $\mathcal{N}_2[\rhos]\to n$ as $\tau\to\infty$.
\end{theorem}

\begin{proof}[Proof of \Cref{thm:convergence to the Gaussian L1}]
To begin with, note that the same argument as in the first half of the proof of \Cref{n1_conv} gives 
\[
\sup_{\tau>1}\|\rhos(\tau)\|_{L^\infty}\leq C\quad\text{ and } \quad\sup_{\tau>1}\|\rhos(\tau)\|_{H^1}\leq C,
\]
where the constant $C$ again depends on $\rho_0$ and $W$ in a quite non-explicit manner. Note that we only need $ W \in \mathcal W^{1,\infty} (\Rd)$, $\nabla W\in L^n(\Rd)$, and $\Delta W \in L^\frac{n}{2}(\Rd)$ (for $n\geq 2$) to get these bounds; in particular they do not rely on the extra assumptions in parts
(a,b,c).

Next we will prove part (a) by obtaining decay estimates for $J_1$ and $J_2$ for $\tau>1$, under the additional assumption that $(-\Delta )^{\frac{1}{2}-\ee}W \in L^1(\mathbb{R})$ for some $\ee \in (0,\frac{1}{2})$. We start with controlling $J_1$ as in the first line of \eqref{J1_1d}, which yields $|J_1|\leq C\|\nabla\Ws*\rhos\|_{L^2}$, thus 
\begin{equation} \label{J1-1d}
	|J_1| \leq C\|\nabla\Ws*\rhos\|_{L^2} \le C\|(-\Delta)^{\frac12-\ee} \Ws \|_{L^1} \| (-\Delta)^{\ee}\rhos \|_{L^{2}} 
	\le C \|(-\Delta)^{\frac12-\ee} W \|_{L^1}  e^{-2\ee\tau},
\end{equation}
where in the last inequality we used \eqref{eq:scaling classical Sobolev}, and that $\sup_{\tau>1}\|\rhos(\tau)\|_{H^1}\leq C$.
For $J_2$,  we have
\[
|J_2|\le \|\nabla \rhos\|_{L^2} \| \nabla \Ws * \rhos \|_{L^2} \leq C \|(-\Delta)^{\frac12-\ee} W \|_{L^1}  e^{-2\ee\tau},
\]
where we used the $J_1$ estimate to control $\|\nabla\Ws*\rhos\|_{L^2}$, and we also used that $\sup_{\tau>1}\|\rhos(\tau)\|_{H^1}\leq C$.
Plugging these into \eqref{eq:E1 general estimate 1} gives \eqref{eq_e1_1d}. The decay estimate for $|\mathcal{N}_2[\rhos]-n|$ follows in the same way as the last paragraph of the proof of  \Cref{n1_conv}.

We now move on to part (b), under the assumption that
$\nabla W\in L^{p_1}(\Rd)$ for some $p_1\in[1,2)$. Again, using that $|J_1|\leq C\|\nabla\Ws*\rhos\|_{L^2}$, for $p_1 \in [1,2)$ we have
\begin{equation} \label{J1-0}
	|J_1| \le C \| \nabla \Ws * \rhos \|_{L^2}  \le C\| \nabla \Ws \|_{L^{p_1}} \| \rhos \|_{L^{q_1}} 
	\le C \| \nabla W \|_{L^{p_1}}  e^{(1-\frac {n}{p_1})\tau}.
\end{equation}
For $J_2$, taking $q_1 = \frac{2p_1}{3p_1 - 2} \in (1,2]$ 
we have
\[
|J_2|\le \|\nabla \rhos\|_{L^2} \| \nabla \Ws * \rhos \|_{L^2} \leq \|\nabla \rhos\|_{L^2} \| \nabla \Ws \|_{L^{p_1}} \| \rhos \|_{L^{q_1}} \le C \| \nabla W \|_{L^{p_1}}  e^{(1-\frac {n}{p_1})\tau},
\]
which has the same decay rate as the $J_1$ estimate.
Plugging these into \eqref{eq:E1 general estimate 1} gives \eqref{eq_e1_temp1}. Again, the decay estimate for $|\mathcal{N}_2[\rhos]-n|$ follows in the same way as the last paragraph of the proof of  \Cref{n1_conv}.

To prove part (c), we start with the $J_1$ estimate. If $p_1\in [1,2)$, the estimate \eqref{J1-0} still holds. And if $p_1\geq 2$, we control $J_1$ as
\begin{equation}\label{J1_temp}
	\left|J_1 \right| \le \|\rhos y\|_{L^{q_1}} \| \nabla \Ws * \rhos \|_{L^{p_1}}  \le \| \rhos \|_{L^{\frac{p_1}{p_1-2}}}^{\frac 1 2} \mathcal N_2 [\rhos] ^{\frac 1 2} \| \nabla \Ws\|_{L^{p_1}} \|\rhos\|_{L^1} \leq C \| \nabla W \|_{L^{p_1}}  e^{(1-\frac {n}{p_1})\tau},
\end{equation}
which gives the same decay rate as \eqref{J1-0}.
For $J_2$ we apply the usual Young inequality
\begin{align*}
	|J_2| \le \|\rhos\|_{L^1} \|\Delta (\Ws * \rhos)\|_{L^\infty} 
	\le C \| \Delta \Ws \|_{L^{p_2}}
	\le C \| \Delta W \|_{L^{p_2}} e^{(2 - \frac n {p_2}) \tau}.
\end{align*}
Plugging these into \eqref{eq:E1 general estimate 1} gives \eqref{eq:convergence to the Gaussian L1 estimate 1}. Again, the decay estimate for $|\mathcal{N}_2[\rhos]-n|$ follows in the same way as above.

Once we finish part (b,c), the last statement in the theorem follows as a direct consequence, since these assumptions of $W$ are covered by part (b) for $n=2$, and part (c) for $n\geq 3$. This finishes the proof.
\end{proof}

The proof of \Cref{thm:decay} follows directly from from \eqref{thm:convergence to the Gaussian L1} and \eqref{CK} (using the Csiszar-Kullback inequality and the change of variables in \eqref{eq:rhos}).

\begin{remark}\label{rmk_conv} Note that the assumptions in \Cref{thm:convergence to the Gaussian L1} allows $W$ to have arbitrarily slow power-law decay at infinity, which is much less restrictive than the $W\in L^1(\Rd)$ assumption in \Cref{n1_conv}. To see this, let $W \in C^\infty(\Rd)$ be a smooth potential with $W = -|x|^{-\ee}$ in $B(0,1)^c$ for some $0<\ee\ll 1$. 
For $n=1$, the definition of fractional Laplacian gives $(-\Delta)^{\frac{1}{2}-\delta}W\sim -|x|^{-\epsilon+2\delta-1}$ for $|x|>1$, thus $(-\Delta)^{\frac{1}{2}-\delta}W \in L^1(\mathbb{R})$ for $\delta\in(0,\frac{\epsilon}{2})$, which satisfies the assumptions in \Cref{thm:convergence to the Gaussian L1}(a).
For $n\geq 2$, one can easily check that $\nabla W \in L^{p_1}(\Rd)$ for all $p_1 > \frac{n}{1+\ee}$ and $\Delta W\in L^{p_2}(\Rd)$ for all $p_2>\frac{n}{2+\ee}$, thus there exists $p_1\in(\frac{n}{1+\ee},n)$ (and $p_2\in(\frac{n}{2+\ee},\frac{n}{2})$ if $n\geq 3$) that satisfy the assumptions in \Cref{thm:convergence to the Gaussian L1}(b,c).
Applying \Cref{thm:convergence to the Gaussian L1} gives $E_1 (  \rhos \| G  )\to 0$ for any $\ee>0$, although the decay rate goes to 0 as $\ee\to 0$. 

From the above example, the assumptions on $W$ in \Cref{thm:convergence to the Gaussian L1} is sharp in the sense that $W=\log|x|$ is the $\ee\to 0$ limit of $W_\ee=\frac{-|x|^{-\ee}+1}{\ee}$, but for such $W$ 
(even if we modify it to be smooth near the origin) it is well-known that the steady state for the rescaled equation \eqref{eq:Fokker-Plack rho scaled} is different from the Gaussian, thus $E_1 (  \rhos \| G  )$ has no decay as $\tau\to\infty$. For this reason, as $p_1$ and $p_2$ approach $n$ and $\frac{n}{2}$ respectively in \Cref{thm:convergence to the Gaussian L1}, it is natural to expect that the convergence becomes arbitrarily slow.
\end{remark}

\subsection{\texorpdfstring{$L^2$}{L2} relative entropy }
\label{sec:L2 rel entropy}
We also look at the convergence of the $L^2$ relative entropy under different assumptions on the interaction potential. 
In order to study the $L^2$ convergence, we define the $L^2$ relative entropy as
\begin{equation*}
	E_2 (  \rhos \| G  ) = 	\int_{\mathbb R^n} \left|\rhos - G \right|^2 G^{-1} \diff y= \int_{\mathbb R^n}  \left| \frac \rhos G - 1 \right|^2 G  \diff y.
\end{equation*}

	Recall that $G$ solves the stationary Fokker-Planck equation. In fact, notice that we can rewrite \eqref{eq:Fokker-Plack rho scaled} as
	\begin{equation*}
		\frac{\partial \rhos}{\partial\tau} = \diver \left(  G \nabla \frac \rhos G  +  \rhos \nabla \Ws * \rho  \right) .
	\end{equation*}
	It is natural that $\rhos / G$ will provide good estimates. In fact, it well-known that the space $L^2 (G^{-1} \diff y)$ is natural because it makes the Fokker-Planck operator self-adjoint. 
	Notice that that $G^{-1} \ge (2\pi)^{\frac n 2} > 0$ so the $L^2(G^{-1}\diff y)$ convergence
	is stronger than the usual $L^2$. 
\begin{theorem}
	\label{thm:convergence to the Gaussian L2}Let $n\geq 2$.
	Let $\rhos$ be a classical solution of \eqref{eq:Fokker-Plack rho scaled} for $\tau \ge 1$
	such that 
	\begin{equation*}
	    \sup_{\tau \ge 1} \|\rhos \|_{L^\infty}  < + \infty
	\end{equation*}
	and $\nabla W \in  L^1 (\mathbb R^n) $.
	Then, 
	\begin{equation*}
			E_2 (  \rhos \| G  )  \le 
			\begin{dcases} 
			    C e^{-2\tau} & n \ge 3,\\ 
			    C (1+\tau)^2 e^{-2\tau}  & n = 2.
			\end{dcases}
	\end{equation*}
	as $\tau \to +\infty$.
	In particular, if $W$ satisfies $ W \in \mathcal W^{1,\infty} (\Rd)$ with $\nabla W \in L^1(\Rd)$ and $ \rho_0\in L^1_+(\mathbb{R}^n)$ satisfies $\int_{\mathbb{R}^n}\rho_0 \diff x = 1$, $E[\rho_0]<\infty$, and $\mathcal{N}_2[\rho_0]<\infty$, then the solution constructed in \Cref{thm:existence} is such that $E_2(\rhos \| G) \to 0$ with the above rates.
\end{theorem}
\begin{remark}
Notice that in this setting we do not use the uniform-in-time bound of $\mathcal N_2 [\rhos]$, nor integrability of $D^2 W$.
\end{remark}

\begin{proof}[Proof of \Cref{thm:convergence to the Gaussian L2}]
We have
\begin{align*}
	\frac 1 2 \frac{\diff }{\diff \tau} \int_{\mathbb R^n}  \left| \frac \rhos G - 1 \right|^2 G  \diff y 
	&= %
	 -  \int_{\mathbb R^n}  \nabla \frac \rhos G   \cdot   \left(  G \nabla \frac \rhos G  +  \rhos \nabla \Ws * \rho  \right) \diff y\\
	& = -  \int_{\mathbb R^n}  \left|  \nabla \frac \rhos G \right|^2   G  \diff y + \int_\Rd   \rhos  \nabla \frac \rhos G \cdot \nabla \Ws * \rho   \diff y .
\end{align*}
We point out that
\begin{align*}
    E_2 (\rhos \| G ) &= \int_\Rd \left| \frac \rhos {\sqrt G} - \sqrt{G}  \right|^2 dy 
    =  \int_\Rd \left| \frac \rhos {\sqrt G}   \right|^2 dy - 1. 
\end{align*}
We now write
\begin{align*}
	\left|  \int_\Rd   \rhos  \nabla \frac \rhos G \cdot \nabla \Ws * \rho   \diff y \right|
	&\le  \int_\Rd   \frac{  \rhos } {\sqrt G } \sqrt G \left|   \nabla \frac \rhos G \right| \left|  \nabla \Ws * \rho  \right|  \diff y 
	\le \left \| \frac{ \rhos }{\sqrt G } \right \|_{L^2} \left\| \sqrt G  \nabla \frac \rhos G  \right \| _{L^2} \| \nabla \Ws *\rhos \|_{L^\infty}  \\
	&= ( E_2 (\rhos \| G) + 1 )^{\frac 1 2} \left\| \sqrt G  \nabla \frac \rhos G  \right \| _{L^2} \| \nabla \Ws \|_{L^1} \| \rhos \|_{L^\infty}.
\end{align*}
Hence
\begin{align}
\label{eq:thm 7.6 ODE 2}
	\frac 1 2 \frac{\diff }{\diff \tau}   E_2 (\rhos \| G)  & \le -  \int_\Rd   \left|  \nabla \frac \rhos G \right|^2   G  \diff y +  ( E_2 (\rhos \| G) + 1 )^{\frac 1 2} \left\| \sqrt G  \nabla \frac \rhos G  \right \| _{L^2} \| \nabla \Ws \|_{L^1} \| \rhos \|_{L^\infty}. 
\end{align}
The last term converges to zero for all $n>1$, because of the scaling
$
	\| \nabla \Ws  \|_{L^1} \le e^{ (1 - n) \tau } \| \nabla W \|_{L^1}.
$
Let us define $w = \frac \rhos G$. In the rescaled heat equation, this converts the Fokker-Planck into the Ornstein-Uhlenbeck semigroup. We recall the Gaussian Poincaré inequality
\begin{equation} \label{GP}
\int_{\mathbb R^n}  \left| w  - 1 \right|^2 G  \diff y =	\int_\Rd |w|^2 G \diff y - \left(  \int_\Rd w G \diff y \right) ^2 \le \int_\Rd |\nabla w|^2 G \diff y\,,
\end{equation}
noticing that $G$ and $\rhos = w G$ have mass equal to 1.
To simplify the notations, let 
\[
u(\tau) :=E_2 (\rhos \| G), \quad v(\tau) := \int_\Rd   \left|  \nabla \frac \rhos G \right|^2   G  \diff y,
\]
and under these notations \eqref{GP} becomes $ 0\le u \le v$. We can also rewrite \eqref{eq:thm 7.6 ODE 2} in these new notations as
\[
\frac{\diff }{\diff \tau} u \le -2 v + C e^{(1-n)\tau} (u+1)^{\frac 12} v^{\frac 12}. 
\] 
Let $\tau_0>1$ be such that $Ce^{(1-n)\tau_0} < 1$, and we claim that $\sup_{\tau \ge \tau_0} u(\tau) \leq \max \{1 , u(\tau_0) \}$. In fact, if $u(\tau) \ge 1$ for some $\tau\geq \tau_0$, using the facts that $Ce^{(1-n)\tau}<1$ and $0\le u\le v$, we have $\frac{\diff u}{\diff \tau} \le -2v + (u+1)^{\frac{1}{2}} v^{\frac12} \leq -2v + \sqrt{2}v< 0$, proving the claim.
Using this estimate, we have that
\[
\frac{\diff }{\diff \tau} u \le -2 v + C e^{(1-n)\tau} v^{\frac 12}\quad
\text{ for all }\tau>\tau_0. 
\]

If $n= 2$, we isolate the $v^{\frac 12}$ in the above last term and use Young's inequality to obtain the following for 
$\tau>\tau_0$:
\[
\frac{\diff }{\diff \tau} u \le -\big(2 - \frac 1{1+\tau}\big) v + C e^{-2\tau} (1+\tau).
\]
Applying the inequality $0\leq u\leq v$ to the right hand side, and multiplying both sides of the inequality by the obvious integrating factor $A(\tau) = \frac {e^{2\tau}}{1+\tau} $, we have 
\[
\frac{\diff }{\diff \tau} (A u) \le C \ \Longrightarrow u(\tau) \le C A(\tau)^{-1} (1+ \tau) \le C (1+\tau)^2 e^{-2\tau}. 
\]

If $n\ge 3$, we proceed slightly differently in isolating $v$ using Young's inequality again
\[
\frac{\diff }{\diff \tau} u \le -(2 - e^{-\tau}) v + C e^{(3 -2n)\tau}.
\]
Using the corresponding integrating factor $A(\tau) = e^{2\tau + e^{-\tau}}$ we obtain 
\[
\frac{\diff }{\diff \tau} (A u) \le C e^{(5-2n)\tau}  \ \Longrightarrow u(\tau) \le C A(\tau)^{-1} \le C  e^{-2\tau}. 
\]
This completes the proof.

\end{proof}

\begin{remark}
	We point out that the $L^2$ norm of $\rho - U$ decays faster than that of $\rho$ itself.
	Using the change of variables \eqref{eq:rhos}, we can translate the result above to
	\begin{equation*}
			\int_{\mathbb R^n}  \left|  {\rho} - U \right|^2  \diff x = (2t + 1)^{- \frac n 2} \int_{ \Rd } |\rhos - G|^2 \diff y \le 
			\begin{dcases}
				C t^{-\frac n 2 - 1} &  n \ge 3, \\
				C t^{- \frac n 2 -1 } (1 + \log(1+2t))^2  & n = 2.
			\end{dcases}
	\end{equation*}
	On the other hand, the $L^2$ decay of $\rho$ itself is only
	\begin{equation*}
		\int_{\mathbb R^n} \rho^2  \diff x = (2t + 1)^{-\frac n 2} \int_{\mathbb R^n} \rhos^2 \diff y    \sim C t ^{-\frac n 2}.
	\end{equation*}
\end{remark}

\appendix
\section{Some comments on fractional Sobolev spaces}
\label{sec:fractional Young-Sobolev inequality}
We recall the definition of the Sobolev-Slobodecki semi-norm for $s \in (0,1)$ below. For $p\in[1,\infty)$, let us define
\begin{equation*}
    [f]_{\mathcal W^{s,p}} = \left( s(1-s) \int_\Rd \int_\Rd \frac{|f(x) - f(y)|^{p}}{|x-y|^{n+sp}} \diff x \diff y \right)^{\frac 1 p},
\end{equation*}
and for $p=\infty$ let
\[
[f]_{\mathcal{W}^{s,\infty}} = \sup_{x\neq y} \frac{|f(x)-f(y)|}{|x-y|^s}.
\]
For $s \in (k,k+1)$ we define
\begin{equation*}
    [ u ]_{\mathcal W^{s,p}}^p =  \sup_{|\alpha|=k} [D^\alpha u]_{\mathcal W^{s-k,p}}. 
\end{equation*}
The complete norm is constructed via
\begin{equation*}
    \|u\|_{\mathcal W^{s,p}}^p = \| u \|_{\mathcal W^{\lfloor s \rfloor, p} }^p + [ u ]_{\mathcal W^{s,p}}^p.
\end{equation*}
A discussion on these norms can be found in \cite{dT+GC+V2020Taylor,Brasco+GC+Vazquez2020}.  
When $s \notin \mathbb N$ the Sobolev-Slobodecki spaces also coincide with the Besov spaces $B_{p,p}^s (\Rd) = \mathcal W^{s,p} (\Rd)$.

\subsection{Scaling of fractional Sobolev norm}
\label{sec:scaling of fractional Sobolev norm}
It is a direct computation that, if $s \in (0,1)$
\begin{align*}
    [ f(\lambda \cdot) ]_{\mathcal W^{s,p}}^p &= s(1-s)\int_\Rd \int_\Rd \frac{|f(\lambda x) - f(\lambda y)|^p}{|x-y|^{n+sp}} \diff x \diff y \\
    &= s(1-s) \int_\Rd \int_\Rd \frac{|f( x) - f( y)|^p}{|\frac x  \lambda - \frac y  \lambda|^{n+sp}} \lambda^{-2n} \diff x \diff y \\
    & = \lambda^{sp-n} [f]_{\mathcal W^{s,p}}^p.
\end{align*}
Hence, for $s \in (0,1)$, we have that
\begin{equation*}
    [ f(\lambda \cdot) ]_{\mathcal W^{s,p}} = \lambda^{s-\frac n p} [f]_{\mathcal W^{s,p}}.
\end{equation*}
Therefore, still for $s \in (0,1)$, we conclude that
\begin{equation*}
    [D^\alpha (f (\lambda \cdot)) ]_{\mathcal W^{s,p}} = [ \lambda^{|\alpha|} ( D^\alpha f ) (\lambda \cdot)  ]_{\mathcal W^{s,p}} = \lambda^{|\alpha|+s-\frac{n}{p}} [D^\alpha f]_{\mathcal W^{s,p}}.
\end{equation*}

\subsection{A result for fractional Laplacians}
The fractional Laplacian is defined through the Fourier transform as the operator of symbol $|\xi|^{2s}$
\begin{equation*}
    (-\Delta)^s u (x) = \mathcal F^{-1} [|\xi|^{2s} \mathcal F[u] (\xi) ]. 
\end{equation*}
Due to the properties of the convolution
\begin{equation*}
    -\Delta (f * g) = [(-\Delta)^{1-s}f ] * [(-\Delta)^s g].
\end{equation*}
Through the standard Young inequality we have that
\begin{equation}
    \label{eq:Young Laplacian}
    \| \Delta (f * g) \|_{L^r} \le \| (-\Delta)^{1-s}f \|_{L^p} \| (-\Delta)^s g \|_{L^q} , \qquad s \in (0,1), 1 + \frac 1 r = \frac 1 p + \frac 1 q.
\end{equation}
Furthermore, it is known (see \cite[Proposition 2.1.7 and Proposition 2.1.8]{Silvestre2005}) that
\begin{equation}
    \label{eq:fLap in Holder}
    \| (-\Delta)^{s} f \|_{C^\ee } \le C \| f \|_{C^{2s+\ee}},
\end{equation}
whenever $\ee, 2s+\ee \notin \mathbb N$.

\subsection{A result in norms}
\begin{theorem}
\label{thm:fractional Young inequality}
Let $s_0, s_1 \ge 0$, and $p_0, p_1 \in [1,\infty]$. Then, there exists $C$ such that
\begin{equation}
\label{eq:fractional Young inequality}
\begin{aligned} 
    \| f*g \|_{\mathcal W^{ s_0 + s_1  ,p }} \le C & \| f \|_{\mathcal W^{s_0,p_0}} \| g \|_{\mathcal W^{s_1,p_1}}, \qquad \tfrac 1 {p} + 1 = \tfrac {1} {p_0} + \tfrac 1 {p_1}.
\end{aligned}
\end{equation}
\end{theorem}
We will use $K$-interpolation. We introduce some definitions and results that can be found in \cite[Chapter 5]{Bennet+Sharpley1988}.

\paragraph{$K$-interpolation.} 
Given $X_0, X_1$ we say they are compatible spaces if they can be embedded into a common Hausdorff topological space $Z$. We define, for $p \in [1,\infty)$
\begin{equation*}
\vertiii{u}_{\theta,p;X_0,X_1} \defeq \left( \int_0^{+\infty} \left( \frac{ K(t,u; X_0, X_1)}{t^{\theta}} \right)^p \frac{\diff t}{t} \right)^{\frac 1 p}
\end{equation*}
and
\begin{equation*}
   \vertiii{u}_{\theta,\infty;X_0,X_1} = \sup_{t > 0} t^{-\theta} K(t,u; X_0, X_1),
\end{equation*}
where
\begin{equation}\label{eq:DefKline}
	K(t,u; X_0, X_1) = \inf  \left\{   \|u_0\|_{X_0} +  t \|u_1 \|_{X_1} : u = u_0 + u_1, u_i \in X_i \right\}.
\end{equation}
We define
\begin{equation*}
    (X_0, X_1)_{\theta,p} = \{ u \in X_0 + X_1 :  \vertiii{u}_{\theta,p;X_0,X_1} < \infty \} .
\end{equation*}
\paragraph{Operators in interpolation spaces} 
The key result we will use is that if for compatible pairs $(X_0,X_1)$ and $(Y_0, Y_1)$ and an operator $T: X_i \to Y_i$ for both $i = 0,1$ we have that $T : (X_0, X_1)_{K;p,\theta} \to (Y_0, Y_1)_{K;p,\theta}$ and we have
\begin{equation*}
    \sup_{u \ne 0} \frac{ \vertiii{ T u }_{ \theta, p;  Y_0, Y_1} }{\vertiii{ u }_{ \theta, p;  X_0, X_1}} \le \left(\sup_{u \ne 0} \frac{ \| T u \|_{Y_0} }{\|u \|_{X_0} } \right)^{1-\theta} \left( \sup_{u \ne 0} \frac{ \| T u \|_{Y_1} }{\|u \|_{X_1} } \right)^{\theta} 
\end{equation*}
This is proved in \cite[Theorem 1.12, Chap 5]{Bennet+Sharpley1988}. 

\paragraph{Embedding.} We will also use the embedding formula \cite[Proposition 1.10, Chap 5]{Bennet+Sharpley1988}, which says that
\begin{equation*}
    (X_0, X_1)_{\theta, q} \subset (X_0, X_1)_{\theta,r}, \qquad \theta \in (0,1), 1 \le q \le r \le \infty.
\end{equation*}
Hence, there exists $C_r(\theta, q, r, X_0, X_1)$ such that
\begin{equation}
    \label{eq:interpolation embedding}
    \vertiii{u}_{\theta,r;  X_0, X_1} \le C_i(\theta, q, r, X_0, X_1) \vertiii{u}_{\theta,q;  X_0, X_1}, \qquad \theta \in (0,1), 1 \le q \le r \le \infty.
\end{equation}

\paragraph{Interpolation reiteration.}
An important result  \cite[Theorem 2.4, Chap 5]{Bennet+Sharpley1988} states that taking interpolation of interpolations is, in itself, and interpolation of the original spaces. Let $0 \le \theta_0 <\theta \le 1 $ and $q, q_0, q_1 \in [1,\infty]$
\begin{equation}
\label{eq:intepolation reiteration}
    \Bigg( (X_0, X_1)_{\theta_0, q_0}, (X_0, X_1)_{\theta_1, q_1} \Bigg)_{\theta,q} = (X_0, X_1)_{\theta', q}, \qquad \theta' = (1-\theta)\theta_0 + \theta \theta_1.
\end{equation}

\paragraph{Interpolation of Sobolev spaces.} According to \cite[Theorem 4.17, Chap 5]{Bennet+Sharpley1988}, if $s_0 \ne s_1$, $\theta \in (0,1)$ and $p,q \in [1,\infty]$ we have
\begin{equation*}
    (\mathcal W^{s_0,p} (\Rd), \mathcal W^{s_1,p} (\Rd) )_{  \theta, q  } = B_{s_\theta,q}^p (\Rd).
\end{equation*}
Here and below
$
    s_\theta = (1-\theta) s_0 + \theta s_1.
$
In particular, we obtain that
\begin{equation*}
    (\mathcal W^{s_0,p} (\Rd), \mathcal W^{s_1,p} (\Rd) )_{ \theta, p  } = \mathcal W^{s_\theta,p} (\Rd), \qquad s_0 \ne s_1 , \theta \in (0,1), p \in [1,\infty], s_\theta \notin \mathbb N.
\end{equation*}
This can be computed from the interpolation between $(L^p, \mathcal W^{k,p})$ and the reiteration formula \eqref{eq:intepolation reiteration}.
To be precise, this means that
\begin{equation}
\label{eq:interpolation norm equivalence}
\begin{aligned}
    \frac{1}{C_e(p,s_0,s_1,\theta)} \| u \|_{\mathcal W^{s_\theta,p}} \le &\vertiii{u}_{ \theta, p;  \mathcal W^{s_0,p}, \mathcal W^{s_1,p}} \le {C_e(p,s_0,s_1,\theta)} \| u \|_{\mathcal W^{s_\theta,p}}, \\
    &\qquad \theta \in (0,1), s_0 \ne s_1, s_{\theta} \notin \mathbb N, p \in [1,\infty].
\end{aligned}
\end{equation}
\paragraph{Proof of \Cref{thm:fractional Young inequality}.}
\textbf{Step 1. $s_0=k_0,s_1=k_1 \in \mathbb N$.}
    Due to the standard Young inequality, we know that
\begin{equation*}
    \| f* g \|_{L^{p}} \le \|f \|_{L^{p_0}}\|g \|_{L^{p_1}}.
\end{equation*}
Applying the result for derivatives, let us write $s = k = k_0 + k_1$ and $\alpha = \alpha_0 + \alpha_1  $
\begin{align*}
    [f*g]_{\mathcal W^{k,p}} &= \sup_{|\alpha|=k} \left\|D^\alpha (f*g) \right\|_{L^{p}} \le \sup_{|\alpha|=k}\| D^{\alpha_0} f \|_{L^{p_0}}\left\|D^{\alpha_1}g\right\|_{L^{p_1}} 
    = [f]_{\mathcal W^{k_0,p_0}} [g]_{\mathcal W^{k_1,p_1}}.
\end{align*}
Thus, we have as expected that
\begin{equation*}
    \| f * g \|_{\mathcal W^{k,p}} \le \| f \|_{\mathcal W^{k_0,p_0}} \| g \|_{\mathcal W^{k_1,p_1}}. 
\end{equation*}

\noindent \textbf{Step 2. $s_0=0, 0 < s_1 \notin \mathbb N$.}
We define the map
$
    T_f : g \mapsto f*g.
$
If $s_1 < k_1 \in \mathbb N$ and $s_1 = (1-\theta)0 + \theta k_1$, then
\begin{equation*}
    \frac{ \| T_f g \|_{L^{p}} }{\|  g \|_{L^{p_1}}} \le \| f \|_{L^{p_0}}, \qquad \text{ and } \qquad \frac{ \| T_f g \|_{\mathcal W^{k_1,p}} }{\|  g \|_{\mathcal W^{k_1,p_1}}} \le \| f \|_{L^{p_0}}.
\end{equation*}
By interpolation, we get
\begin{equation*}
    \frac{ \vertiii{T_f g}_{ \theta, p;  L^p, \mathcal W^{k_1,p}} }{\vertiii{g}_{ \theta, p;  L^{p_1}, \mathcal W^{k_1,p_1}}} \le \| f \|_{L^{p_0}}
\end{equation*}
Notice that the interpolation for $g$ is with $p$, and not $p_1$ as we would like. Using the embedding formula \eqref{eq:interpolation embedding} since $p \ge p_1$ we have
\begin{equation*}
    { \vertiii{f*g}_{ \theta, p;  L^p, \mathcal W^{k_1,p}} } \le C_r(\theta,p_1,p,L^p, \mathcal W^{k_1,p}) \| f \|_{L^{p_0}} {\vertiii{g}_{ \theta, p_1;  L^{p_1}, \mathcal W^{k_1,p_1}}}.
\end{equation*}
Using \eqref{eq:interpolation norm equivalence} we deduce the result, multiplying the constant $C_e$ to our right-hand side.

\noindent \textbf{Step 3. $s_0 \in \mathbb N, 0 < s_1 \notin \mathbb N$.} We proceed as in Step 1, followed by Step 2.

\noindent \textbf{Step 4. $s_0, s_1 \notin \mathbb N$.} Now we must interpolate in $f$. For $g$ fixed we define $T_g : f \mapsto f * g$. 
Let $s_0 < k_0 \in \mathbb N$ and $\theta$ such that $s_0 =  (1-\theta)0 + \theta k_0$. 
By Steps 2 and 3 we have that
\begin{equation*}
    \frac{ \| T_g f \|_{\mathcal W^{s_1,p} }}{\|  f \|_{L^{p_0}}} \le C\| g \|_{\mathcal W^{s_1,p_1}}, \qquad \text{ and } \quad \frac{ \| T_g f \|_{\mathcal W^{k_0 + s_1,p}} }{\|  f \|_{\mathcal W^{k_0,p_0}}} \le C\| g \|_{\mathcal W^{s_1,p_1}}.
\end{equation*}
Hence, by interpolation we obtain that
\begin{equation*}
    \frac{ \vertiii{T_g f}_{ \theta, p;  \mathcal W^{s_1,p}, \mathcal W^{k_0+s_1,p}} }{\vertiii{f}_{ \theta, p;  L^{p_0}, \mathcal W^{k_0,p_0}}} 
    \le C \| g \|_{\mathcal W^{s_1,p_1}}.
\end{equation*}
We apply again the embedding \eqref{eq:interpolation embedding} and \eqref{eq:interpolation norm equivalence} to conclude the result.\qed

\section{Relating \texorpdfstring{$\rho \log \rho$}{rho log rho} and the second moment}
We provide a general result relating moment-type bounds and the integrability of $F(\rho)$ with the integrability of $|F(\rho)|$. This is very useful to obtain equi-integrability results. In the case of $F(s) = s\log s $ and the second moment, this is very classical (see, e.g. \cite{BD95,BCC12}).
\begin{lemma}
\label{lem:rho log rho + second moment implies rho | log rho |}
    Let 
    $F : \mathbb R \to \mathbb R$  continuous be such that:
    \begin{enumerate} 
    \item For some $\rho_1, \rho_2 > 0$
    $
        |F| \text{ is non-decreasing in }  [0, \rho_1] $, $F\ge 0 \text{ in } [\rho_2 , +\infty)
    $
    \item There exists $g$ non-decreasing and locally integrable such that
      $
        \rho \in L^1 (\Rd)$, 
    \begin{equation*} \mathcal N_g [\rho] = \int_\Rd g(|x|) \rho(x) < \infty
    \qquad 
    \text{ and define }
    G(s) = \int_0^s g(r) r^{n-1} \diff r.
    \end{equation*}
    \end{enumerate}
	Then, there exists $C_0, r_0$ depending only on $\rho_1, \rho_2$ and $\| \rho \|_{L^1}$ such that
	\begin{equation*}
	    \int_\Rd  |F(\rho)|  \le \int_{\Rd} F(\rho)  + C_0 \left(  \sup_{ \rho_1 < s < \rho_2} |F(s)| + \int_{|x| > r_0} \left| F \left(  \frac{\mathcal N_g [\rho]  }{|S^{n-1}|G(|x|)}  \right)   \right|  \right).
	\end{equation*}
\end{lemma}

\begin{proof}
	We decompose the $F$ into its positive and negative parts
	$
		F = F_+ - F_- .
	$
	We consider $\rho^*$ the decreasing rearrangement of $\rho$. We have that
	\begin{equation*}
		\int_\Rd F_\pm ( \rho )  = \int_\Rd F_\pm   ( \rho^* ) , \qquad N_g [\rho^* ] = \int g(|x|) \rho^* \le \int g(|x|) \rho = N_g [ \rho ]. 
	\end{equation*}
	First, let us estimate $\int F_- (\rho)$. 
	Since $\rho^*$ is decreasing, we have can apply Lieb's trick (see \cite{Lieb1983}) to show that 
	\begin{equation*}
		\mathcal N_g [\rho] \ge \int_\Rd g( |x| ) \rho^* \ge |S^{n-1}| \int_{0}^{|x|} g(r) \rho^* r^{n-1}\diff r \ge |S^{n-1}|\rho^* (x) G(|x|),
	\end{equation*}
	where $|S^{n-1}|$ is the measure of the $(n-1)$-dim sphere. 
	Thus, we have that
	\begin{equation*}
		\rho^* (x) \le  \frac{N_g [\rho]} {|S^{n-1}|G(|x|)} .
	\end{equation*}
	Since $\rho^*$ is decreasing and tends to $0$ there exists $r_1, r_2 > 0$ such that
	$
	    \overline {\{ \rho^* > \rho_i \}} = \overline{ B_{r_i}}.  
	$
	Notice that
	\begin{equation*}
	    \int_{\Rd} \rho = \int_\Rd \rho^* \ge \int_{\rho^* > \rho_i} \rho^* \ge \rho_i |\{ \rho^* > \rho_i   \}| = \rho_i | B_1| r_i^n.
	\end{equation*}
	Thus, we can estimate
	\begin{equation*}
	    r_i \le \left( \frac{\int_\Rd \rho }{ \rho_i | B_1| } \right)^{\frac 1 n}.
	\end{equation*}
	We take 
	\[
	r_0 = \max \left\{\left( \frac{\int_\Rd \rho }{ \rho_1 | B_1| } \right)^{\frac 1 n}, \ \left( \frac{\int_\Rd \rho }{ \rho_2 | B_1| } \right)^{\frac 1 n} \right\}.
	\]
	We have that
	\begin{align*}
		\int_\Rd F_- (\rho^*)  &=   \int_{|x| > r_2} F_- (\rho^*)  
		\le  \int_{|x| > r_2} | F (\rho^*) |  
		\le \int_{ r_2 < |x| < r_0} | F (\rho^*) |  + \int_{|x| > r_0} | F (\rho^*) |  \\ 
		&\le   |B_{r_0} \setminus B_{r_2}| \sup_{ r_2 < |x| < r_0} |F (\rho^*)|  + \int_{|x| > r_0} \left| F \left ( \frac{N_g [\rho]} {|S^{n-1}|G(|x|)} \right) \right| .
	\end{align*} 
	Since $F_+ = F + F_-$ we have that $|F| = F + 2 F_-$ and this proves the result.
\end{proof}

\section*{Acknowledgments}
The research of JAC and DGC was supported by the Advanced Grant Nonlocal-CPD (Nonlocal PDEs for Complex Particle Dynamics:
Phase Transitions, Patterns and Synchronization) of the European Research Council Executive Agency (ERC) under the European Union’s Horizon 2020 research and innovation programme (grant agreement No. 883363).
JAC was partially supported by EPSRC grant number EP/T022132/1.
The research of DGC was partially supported by grant PGC2018-098440-B-I00 from the Ministerio de Ciencia, Innovación y Universidades of the Spanish Government. 
The research of YY was partially supported by grants NSF DMS 1715418, 1846745 from the National Science Foundation of the US, and the Sloan Research Fellowship.
The research of CZ was partially supported by grant NSF DMS 19000083 from the National Science Foundation of the US. 
The problem in this paper was discussed at the AIM workshop ``Nonlocal differential equations in collective behavior'' in Jun 2018. JAC and YY thank AIM for support and collaborative opportunity, and thank the workshop participants for stimulating discussions.

\printbibliography

\end{document}